\documentclass{article}
\usepackage[margin=1in]{geometry}  
\usepackage{graphicx}              
\usepackage{amsmath, amssymb, amsthm, epsfig}               
\usepackage{enumerate}
\usepackage{amsfonts}              
\usepackage{amsthm}                
\usepackage{amsthm}
\usepackage{enumerate}
\theoremstyle{plain}
\usepackage{color}
\newtheorem{corollary}{Corollary}[section]
\newtheorem{definition}[corollary]{Definition}

\newtheorem{lemma}[corollary]{Lemma}
\newtheorem{prop}[corollary]{Proposition}
\newtheorem{rem}[corollary]{Remark}
\newtheorem{thm}[corollary]{Theorem}

\newfont{\sBlackboard}{msbm10 scaled 1200}

\newcommand{\mylabel}[1]{\label{#1}
    \ifx\undefined\stillediting
    \else \fbox{$#1$}\fi }
\newcommand{\BE}{\begin{equation}}

\newcommand{\EEQ}{\end{equation}}
\newcommand{\rfb}[1]{\mbox{\rm
        (\ref{#1})}\ifx\undefined\stillediting\else:\fbox{$#1$}\fi}

\newfont{\Blackboard}{msbm10 scaled 1200}

\newfont{\roma}{cmr10 scaled 1200}

\newcommand{\bb}{\begin{equation}}
\newcommand{\bbb}{\end{equation}}

\newcommand{\mm}    {{\hbox{\hskip 0.5pt}}}

\newcommand{\bluff} {{\hbox{\raise 15pt \hbox{\mm}}}}
scaled\magstep2

\makeatletter
\def\section{\@startsection {section}{1}{\z@}{-3.5ex plus -1ex minus
        -.2ex}{2.3ex plus .2ex}{\large\bf}}
\makeatother
\numberwithin{equation}{section}

\begin{document}
\title{Ground state and nodal solutions for fractional Orlicz problems with lack of regularity and without the Ambrosetti-Rabinowitz condition}
\author{Hlel Missaoui\footnote{
hlel.missaoui@fsm.rnu.tn;\ hlelmissaoui55@gmail.com}\ \ and  Hichem Ounaies\footnote{ hichem.ounaies@fsm.rnu.tn}\\
Mathematics Department, Faculty of Sciences, University of Monastir,\\ 5019 Monastir, Tunisia}
\maketitle


\begin{abstract}
We consider a  non-local Shr\"odinger problem driven by the fractional  Orlicz g-Laplace operator as follows
\begin{equation}\label{PP}
(-\triangle_{g})^{\alpha}u+g(u)=K(x)f(x,u),\ \ \text{in}\  \mathbb{R}^{d},\tag{P}
\end{equation}
where $d\geq 3,\ (-\triangle_{g})^{\alpha}$ is the fractional Orlicz g-Laplace
operator, $f:\mathbb{R}^d\times\mathbb{R}\rightarrow \mathbb{R}$ is a measurable function and $K$ is a positive continuous function. Employing the Nehari manifold method and without assuming the well-known Ambrosetti-Rabinowitz  and differentiability conditions on the non-linear term $f$, we prove that the problem \eqref{PP} has a ground state of fixed sign and a nodal (or sign-changing) solutions. 
\end{abstract}



\section{Introduction}
Recently, much attention has been focused on the study of non-linear problems involving non-local operators. These types of operators arise in several areas such as the description of many physical phenomena (see \cite[Part II - Chapters 12 and 13]{41}).

In this paper we consider the following non-local Shr\"odinger  equation
\begin{equation}\label{P}
(-\triangle_{g})^{\alpha}u+g(u)=K(x)f(x,u),\ \ \text{in}\
\mathbb{R}^{d},\tag{P}
 \end{equation}
 where $d\geq 3$, $f:\mathbb{R}^d\times\mathbb{R}\rightarrow \mathbb{R}$ is a measurable function, $K$ is a positive continuous function, and $(-\triangle_{g})^{\alpha}$ is the fractional Orlicz g-Laplace operator introduced in \cite{42} and defined as 
 \begin{align}\label{4eq170}
(-\triangle_{g})^{\alpha}u(x)&= \text{p.v.}\int_{\mathbb{R}^{d}} g\bigg{(}\frac{|u(x)-u(y)|}{|x-y|^{\alpha}}\bigg{)}\frac{u(x)-u(y)}{|u(x)-u(y)|}\frac{dy}{|x-y|^{d+\alpha}},
\end{align}
where p.v. being a commonly used abbreviation for "in the principle value sense", $\alpha\in (0,1)$ and $G$ is an N-function such that $g=G^{'}$. We recall the definition of N-function and its properties later in Section 2. The variational setting for the fractional g-Laplace operator is the fractional Orlicz-Sobolev space $W^{\alpha,G}(\mathbb{R}^d)$, which is introduced in \cite{42}. For more details on the fractional Orlicz-Sobolev spaces, we refer the reader to \cite{442,443,421,440,441,46} and the references therein.\\

 Observe that when $\alpha=1$ and $g(t) = \vert t\vert^{p-2}t$, $p > 1$, the problem \eqref{P} turns into  the classical p-Laplacian problem
 $$-\text{div}(\vert \nabla u\vert^{p-2} \nabla u)+\vert u\vert^{p-2}u=K(x)f(x,u).$$
 When $\alpha=1$ and $g(t)= a(\vert t\vert)t$, the problem \eqref{P}  transmute into  the Orlicz g-Laplacian problem
$$-\text{div}(a(\vert \nabla u\vert) \nabla u)+ a(\vert u\vert)u=K(x)f(x,u),\ \text{where}\ \text{div}(a(\vert \nabla u\vert) \nabla u)\ \text{is the Orlicz g-Laplace operator}.$$
When $\alpha\in (0,1)$ and $g(t) = \vert t\vert^{p-2}t$, the problem \eqref{P} transformed into  the fractional p-Laplacian problem
 $$-(\triangle_p)^\alpha u+\vert u\vert^{p-2}u=K(x)f(x,u),\ \text{where}\ (\triangle_p)^\alpha u\ \text{is the fractional p-Laplace operator}.$$
 
 In the last decades, the existence of ground state and nodal solutions for the above problems (classical p-Laplacian, Orlicz g-Laplacian and fractional p-Laplacian problems) have been studied extensively. We do not intend to review the huge bibliography 
 , we just emphasize that the  Nehari method is a very effective tool for proving the existence of  solutions of such problems, see \cite{43,44,47,424,423,425,426,414,427,46,49,410,411,45,413} and the references therein.\\

In \cite{424}, S. Barile and G. M. Figueredo have studied the following equation
\begin{equation}\label{1.1}
-\text{div}(a(\vert \nabla u\vert^{p})\vert \nabla u\vert^{p-2}\nabla u)+V(x)b(\vert u\vert^{p})\vert u\vert^{p-2}=K(x)f(u),\ \ \text{in}\  \mathbb{R}^{d},\tag{$P_1$}
 \end{equation}
where $d\geq 3,\ 2\leq p<d,\ a,\ b,$ are $C^{1}$ real functions
and $V,\ K$ are continuous positives functions. By assuming the well-Known Ambrosetti-Rabinowitz (AR for short), differentiability ($f\in C^1$) conditions on the non-linear term $f$ and by using a minimization argument coupled with a quantitative deformation lemma, they proved the existence of a least energy sign-changing solution for equation \eqref{1.1} with two nodal domains.\\

In \cite{46}, G. M. Figueredo considered   the following  equation
   \begin{equation}\label{eqf}
-M\left( \int_{\Omega}g(\vert \nabla u\vert)dx\right)\Delta_g u=f(u),\ \text{in}\ \Omega, \tag{$P_2$}
 \end{equation}
where $\Omega$ is a bounded domain in $\mathbb{R}^{d}$, $M$ is  $C^1$ function and $\Delta_gu:=\text{div}(a(\vert \nabla u\vert)\nabla u)$ is the Orlicz g-Laplace operator ($g(t)=a(\vert t\vert)t$). By considering the (AR) and differentiability conditions  on the non-linear term $f$ and following the approaches employed in \cite{424}, the author proved the existence of a least energy nodal solution for equation \eqref{eqf}.\\

 In \cite{44}, V. Ambrosio and  T. Isernia  have studied the following  fractional equation
\begin{equation}\label{eqa}
-(\triangle)^\alpha u+ V(x)u=K(x)f(u),\ \ \text{in}\  \mathbb{R}^{d},\tag{$P_3$}
 \end{equation}
where $\alpha\in (0,1)$ and $d>2\alpha$, $(\triangle)^\alpha u$ is the fractional Laplace operator and $V,\ K$ are continuous functions. Under the (AR) and differentiability conditions on the non-linear term $f$ and by applying a minimization method combined with a quantitative deformation lemma, they proved the existence of a least energy nodal solution. In \cite{47}, the authors extended the result obtained in \cite{44} to a non-local generalized fractional Orlicz equations of type \eqref{P} (under the assumptions considered in \cite{44,424,46} on the non-linear term ).\\

 In \cite{49}, without the (AR) and differentiability conditions on the non-linear term $f$, G. M. Figueiredo and J. R. Santos J\'unior established the existence of least energy sign-changing solutions for problem \eqref{eqf}. Indicate that the authors developed a new approach based on the topological degree. In \cite{43}, by following closely the approach developed in \cite{49} and under the same assumptions on the non-linear term $f$, V. Ambrosio et al. proved the existence of a least energy nodal solution for equation \eqref{eqa}.\\


The goal of this paper is to prove the existence of a ground state and a least energy  nodal weak solution for the generalized fractional Orlicz problem  \eqref{P} without the (AR) and differentiability conditions on the non-linear term $f$. We mean by  weak solution of problem \eqref{P} a critical point $u\in W^{\alpha,G}(\mathbb{R}^d)$ of the energy functional $J$ associated to problem \eqref{P}: $J^{'}(u)v=0,\ \ \text{for all}\ v\in W^{\alpha,G}(\mathbb{R}^d)$ (the functional $J$ will be defined later in Section 4).
Our approaches are based on the Nehari method and the generalized subdifferential. The main features and difficulties  involved in the study  of the problem \eqref{P} are listed as follows:
\begin{enumerate}
\item[$(1)$] The non-local character of the fractional Orlicz g-Laplacian. More precisely, in contrast with the classic case \cite{424,423,425,426,427}, in the fractional Orlicz framework we do not have the following decompositions
 $$J(u) = J(u^+) + J(u^-),\ \text{and}\ J^{'}(u^-)u^- =J^{'}(u^+)u^+ = 0,\text{where}\ u=u^++u^-$$
for $u$ belonging to a subset $\mathcal{M}$ of Nehari manifold $\mathcal{N}$ of $J$ (the definitions of $u^+,u^-$, and the sets $\mathcal{N},\mathcal{M}$ will be specified later in Section 4). Such facts make the use of the minimization method more difficult. To overcome this difficulty, we use a new estimate which is inspired by the work \cite{415}.
\item[$(2)$]Unlike with the hypotheses on the non-linear term  $f$ stated in \cite{44,47,48,415}, in our present work, we do not require the differentiability of $f$ (see hypotheses $(H_f)$ Section 2). So, we do not hope the existence of a differentiable structure in the sets $\mathcal{N}$ and $\mathcal{M}$. For more details about this subject, we refer the reader to \cite{444}. To surmount the lack of differentiability, we use the idea employed in \cite{416}, which is based on the non-smooth multiplier rule of Clarke \cite[Theorem 10.47, p. 221]{418}.
\item[$(3)$]  Since we do not assume the (AR) condition on $f$, it is not trivial that minimizing sequences are bounded. We overcome this difficulty, with adaptation of arguments employed in \cite{411}. Moreover, we invoke  Miranda's theorem \cite{412} to show that $\mathcal{M}\neq \emptyset$.
\end{enumerate}

To the best of our knowledge, there is only two papers in the literature, see \cite{47,48}, devoted for the existence and multiplicity of sign-changing solutions for fractional Orlicz problems under the (AR) condition. There are no results treating the existence and multiplicity of sign-changing solutions for fractional Orlicz problems without the (AR) condition.\\ 

 This paper is organized as follows. In Section 2, firstly, we set the variational framework related to the problem \eqref{P} and we establish some properties about N-functions. Secondly, we present the definition of the generalized subdifferential. In Section 3, we set the hypotheses on the weight function $K$ and the non-linear term $f$, and we state our main result (Theorem \ref{4thm1}). In Section 4, we present the energy functional associated to problem \eqref{P} and we prove some technical lemmas. In Section 5, we show the existence of a ground state solution for problem \eqref{P}. Finally, in Section 6  we prove the existence of a least energy nodal weak solution. 
\section{Mathematical preliminaries}
\subsection{Framework setting: Fractional Orlicz-Sobolev Spaces}
In this subsection, we recall some necessary properties about N-functions and the fractional Orlicz-Sobolev spaces. for more details we refer the reader to \cite{421,422,42,22}.

In order to construct a fractional Orlicz-Sobolev space setting for problem \eqref{P}, we consider the following  assumptions on $G$ and $g$:
\begin{enumerate}
\item[$(H_G)$]\ \ $g:\mathbb{R}\rightarrow \mathbb{R}$ is an odd, continuous and non-decreasing function and  $G:\mathbb{R}\rightarrow \mathbb{R}^+$ defined by
\begin{equation}\label{4eq40}
G(t)=\int_{0}^{ t} g(s)\ ds,
\end{equation}
such that $G$ and $g$ satisfy the following assumptions
\begin{enumerate}
\item[$(g_1)$] $g(t)>0$, for all $t>0$, $g(0)=0$  and $\lim\limits_{t\rightarrow +\infty}g(t)=+\infty.$
\item[$(g_2)$] There exist $g^-,g^+\in (1,d)$ such that
$$g^-\leq \frac{g(t)t}{G(t)}\leq g^+,\ \ \text{for all}\ t>0.$$
\item[$(g_3)$] $g\in C^1(\mathbb{R_+^*})$ and
$\displaystyle{g^--1\leq \frac{g^{'}(t)t}{g(t)}\leq g^+-1,\ \ \text{for all}\ t>0.}$
\item[$(g_4)$] $\displaystyle \int_{0}^{1}\frac{G^{-1}(t)}{t^{\frac{d+\alpha}{d}}}dt<\infty$ and  $\displaystyle
\int_{1}^{+\infty}\frac{G^{-1}(t)}{t^{\frac{d+\alpha}{d}}}dt=\infty$.
\end{enumerate}
\end{enumerate}
 $G$ is an N-function: $G$ is even, positive, continuous and  convex function, Moreover  $\frac{G(t)}{t}\rightarrow 0$ as $t\rightarrow0$ and  $\frac{G(t)}{t}\rightarrow +\infty$ as $t\rightarrow+\infty$ (see \cite{22}).\\
The conjugate N-function of G denoted $\tilde{G}$ is defined by
$$\tilde{G}(t)=\int_{0}^{ t} \tilde{g}(s)\ ds,$$
where  $\tilde{g}: \mathbb{R}\rightarrow\mathbb{R}$ is given by
$\tilde{g}(t)=\sup\{s:\ g(s)\leq t\}$. Since $g$ is a continuous function, it comes that $\tilde{g}(\cdot)=g^{-1}(\cdot)$.\\
 Involving the functions $G$ and $\tilde{G}$, we have the Young’s inequality
\begin{equation}\label{9}
st\leq G(s)+\tilde{G}(t).
\end{equation}
We say that an N-function $G$ satisfies the  $\triangle_2$-condition, if there exists $C>0$ such that
\begin{equation}\label{4eq41}
G(2t) \leq CG(t),\ \text{for all}\  t > 0.
\end{equation}
The assumption $(g_2)$ implies that $G$ and $\tilde{G}$ satisfy the  $\triangle_2$-condition (see \cite{22}).\\
Let $A$ and $B$ are two N-functions, we say that $A$ is essentially
stronger than $B$ ($B\prec\prec A$ in
symbols), if  for every positive constant $k$, we have
$$\lim_{t\rightarrow+\infty}\frac{B(kt)}{A(t)}=0.$$
Another important function related to the N-function $G,$ is the Sobolev conjugate N-function denoted $G_{*}$ and defined by 
$$G_{*}^{-1}(t)=\int_{0}^{t}\frac{G^{-1}(s)}{s^{\frac{d+\alpha}{d}}}\ ds, \ t>0.$$
$\displaystyle{G_{*}(t)=\int_{0}^{t}g_{*}(s)\ ds}$ and according to  assumption $(g_2)$,  
\begin{equation}\label{2.7}
g^-_{*}\leq \frac{g_*(t)t}{G_*(t)}\leq g^+_{*},\ \text{for all} \ t>0,\ \text{where}\ g^+_{*}=\frac{dg^+}{d- g^+}\ \text{and}\ g^-_{*}=\frac{dg^-}{d- g^-}.
\end{equation}

Since $G$ satisfies the $\Delta_{2}$-condition, the Orlicz space $L^{G}(\mathbb{R}^{d})$  is the vectorial space of  all measurable function $u:\mathbb{R}^{d}\rightarrow\mathbb{R}$ satisfies
$$\tilde{\rho}(u):=\int_{\mathbb{R}^{d}}G(u)\ dx < \infty\ (\text{see}\ \cite{22}).$$
$L^{G}(\mathbb{R}^{d})$ is a Banach space under the Luxemburg norm
$$\Vert u\Vert_{(G)}=\inf \left\lbrace  \lambda >0\ : \ \tilde{\rho}(\frac{u}{\lambda})\leq 1\right\rbrace. $$\\
Next, we introduce the  fractional Orlicz-Sobolev space. We denote by $W^{\alpha,G}(\mathbb{R}^{d})$ the fractional Orlicz-Sobolev space defined by
\begin{equation}\label{20}
    W^{\alpha,G}(\mathbb{R}^{d})=\bigg{\{}u\in
L^{G}(\mathbb{R}^{d}):\ \overline{\rho}(\alpha;u)<\infty\bigg{\}},
    \end{equation}
 where $\displaystyle{\overline{\rho}(\alpha;u):=\int_{\mathbb{R}^{d}}\int_{\mathbb{R}^{d}} G\bigg{(}\frac{u(x)-u(y)}{|x-y|^{\alpha}}\bigg{)}\frac{dxdy}{|x-y|^{d}}}.$
    
 The space $W^{\alpha,G}(\mathbb{R}^d)$ is equipped with the norm,
    \begin{equation}\label{21}
    \|u\|_{\alpha,G}=\|u\|_{(G)}+[u]_{(\alpha,G)},
    \end{equation}
    where $[\cdot]_{(\alpha,G)}$ is the Gagliardo semi-norm defined by
    \begin{equation}\label{22}
    [u]_{(\alpha,G)}=\inf\bigg{\{}\lambda>0:\ \overline{\rho}(\alpha;\frac{u}{\lambda})\leq 1\bigg{\}}.
    \end{equation}
    
Since $G$ and $\tilde{G}$ satisfy the $\triangle_{2}$-condition, so the fractional Orlicz-Sobolev space $W^{\alpha,G}(\mathbb{R}^d)$ is a separable and reflexive Banach space. Moreover, $C^{\infty}_{0}(\mathbb{R}^{d})$ is dense in $W^{\alpha,G}(\mathbb{R}^{d})$ (see \cite[Proposition 2.10]{42}).

 Let
 \begin{equation}\label{4eq60}
  \rho(\alpha; u):=\tilde{\rho}(u)+\overline{\rho}(\alpha;u)\ \text{and}\ \Vert u\Vert=\inf\left\lbrace \lambda>0 : \rho(\alpha;\frac{u}{\lambda})\leq 1\right\rbrace. 
   \end{equation}
Evidently, $\Vert u\Vert$ is an  equivalent norm to $\|u\|_{\alpha,G}$ with the relation
 $$\frac{1}{2}\|u\|_{\alpha,G}\leq \Vert u\Vert\leq 2\|u\|_{\alpha,G},\ \text{for all}\ u\in W^{\alpha,G}(\mathbb{R}^{d}).$$
 In the sequel we will use  $\Vert \cdot\Vert$ as a norm for the space $W^{\alpha,G}(\mathbb{R}^{d})$.
 
As we mentioned in Section 1, the fractional g-Laplace operator is defined by
\begin{align}\label{17}
(-\triangle_{g})^{\alpha}u(x)&=\text{p.v.}\int_{\mathbb{R}^{d}} g\bigg{(}\frac{|u(x)-u(y)|}{|x-y|^{\alpha}}\bigg{)}\frac{u(x)-u(y)}{|u(x)-u(y)|}\frac{dy}{|x-y|^{d+\alpha}}\nonumber\\
&= \text{p.v.}\int_{\mathbb{R}^{d}} g\bigg{(}\frac{u(x)-u(y)}{|x-y|^{\alpha}}\bigg{)}\frac{dy}{|x-y|^{d+\alpha}}\ \ (\text{since}\ g\ \text{is odd}).
\end{align}
 $(-\triangle_{g})^{\alpha}$ is well defined between $W^{\alpha,G}(\mathbb{R}^{d})$ and its topological dual space  $(W^{\alpha,G}(\mathbb{R}^{d}))^*=W^{-\alpha,\tilde{G}}(\mathbb{R}^{d})$. In fact, in \cite[Theorem 6.12]{42}, the following representation formula is provided
\begin{equation}\label{18}
\langle(-\triangle_{g})^{\alpha}u,v\rangle=\int_{\mathbb{R}^{d}}\int_{\mathbb{R}^{d}} g\bigg{(}\frac{u(x)-u(y)}{|x-y|^{\alpha}}\bigg{)}\frac{v(x)-v(y)}{|x-y|^{\alpha}}\frac{dxdy}{|x-y|^{d}},
\end{equation}
for all $u, v\in W^{\alpha,G}(\mathbb{R}^{d})$. Where $\langle \cdot , \cdot\rangle$ is the duality brackets for the pair $\left( W^{\alpha,G}(\mathbb{R}^{d}),W^{-\alpha,\tilde{G}}(\mathbb{R}^{d})\right)$.

In what follows, we give some theorems and lemmas related to N-functions and the fractional Orlicz-Sobolev  space.
 \begin{thm}\label{4thm2} $($see \cite{421}$)$\\
 Under the assumption $(H_G)$, the continuous embedding $W^{\alpha,G}(\mathbb{R}^{d})\hookrightarrow L^{B}(\mathbb{R}^{d})$ holds, for all $B$ N-function satisfies the $\triangle_{2}$-condition and $B \prec\prec G_{*}$.
\end{thm}

Under the assumptions  $(g_1)-(g_3)$, some elementary inequalities and properties listed in the following lemmas are valid. See \cite{421,422,46,hlel42}.

\begin{lemma}\label{lem444}
If $G$ is an N-function, then
$$G(a+b)\geq G(a)+G(b),\ \text{for all}\ a,b\geq 0.$$
\end{lemma}

\begin{lemma}\label{lem2.2}
Under the assumptions $(g_1)-(g_3)$, the functions $G$ and $\tilde{G}$ satisfy the following inequalities
\begin{equation}\label{2.6}
\tilde{G}(g(t))\leq G(2t)\ \ \text{and}\ \ \tilde{G}(\frac{G(t)}{t})\leq G(t)\ \ \forall\ t\geq 0.
\end{equation}
\end{lemma}
\begin{lemma}\label{4lem1}
Assume that the assumptions $(g_1)-(g_3)$ hold, then
\begin{enumerate}
\item[$(1)$]$ \min\lbrace a^{g^-},a^{g^+}\rbrace G(t)\leq G(at)\leq \max\lbrace a^{g^-},a^{g^+}\rbrace G(t),\ \ \text{for all}\ a,t\geq 0.$
\item[$(2)$] $ \min\lbrace a^{g^--1},a^{g^+-1}\rbrace g(t)\leq g(at)\leq \max\lbrace a^{g^--1},a^{g^+-1}\rbrace g(t),\ \ \text{for all}\ a,t\geq 0.$
\item[$(3)$] $ \min\lbrace a^{g^-_{*}},a^{g^+_{*}}\rbrace G_*(t)\leq G_*(at)\leq \max\lbrace a^{g^-_{*}},a^{g^+_{*}}\rbrace G_*(t),\ \ \text{for all}\ a,t\geq 0.$
\item[$(4)$] $ \min\lbrace a^{\frac{g^-}{g^--1}},a^{\frac{g^+}{g^+-1}}\rbrace \tilde{G}(t)\leq \tilde{G}(at)\leq \max\lbrace a^{\frac{g^-}{g^--1}},a^{\frac{g^+}{g^+-1}}\rbrace \tilde{G}(t),\ \ \text{for all}\ a,t\geq 0.$
\end{enumerate}
\end{lemma}
\begin{lemma}\label{4lem2}
Assume that the assumptions $(g_1)-(g_3)$ hold, then
\begin{enumerate}
\item[$(1)$] $ \min\lbrace \Vert u\Vert_{(G)}^{g^-},\Vert u\Vert_{(G)}^{g^+}\rbrace \leq \tilde{\rho}(u) \leq \max\lbrace \Vert u\Vert_{(G)}^{g^-},\Vert u\Vert_{(G)}^{g^+}\rbrace,\ \text{for all}\ u \in L^{G}(\mathbb{R}^{d}).$
\item[$(2)$] $ \min\lbrace \Vert u\Vert_{(G_*)}^{g^-_*},\Vert u\Vert_{(G_*)}^{g^+_*}\rbrace \leq \tilde{\rho}(u) \leq \max\lbrace \Vert u\Vert_{(G_*)}^{g^-_*},\Vert u\Vert_{(G_*)}^{g^+_*}\rbrace,\ \text{for all}\ u \in L^{G_*}(\mathbb{R}^{d}).$
\item[$(3)$]
$ \min\lbrace [u]_{\alpha,G}^{g^-},[u]_{\alpha,G}^{g^+}\rbrace \leq \overline{\rho}(\alpha;u) \leq \max\lbrace [u]_{\alpha,G}^{g^-},[u]_{\alpha,G}^{g^+}\rbrace,\ $
$\text{for all}\ u \in W^{\alpha,G}(\mathbb{R}^{d}).$
\item[$(4)$] $ \min\lbrace \Vert u\Vert^{g^-},\Vert u\Vert^{g^+}\rbrace \leq \rho(\alpha;u) \leq \max\lbrace \Vert u\Vert^{g^-},\Vert u\Vert^{g^+}\rbrace,\ \text{for all}\ u \in W^{\alpha,G}(\mathbb{R}^{d})$.
\end{enumerate}
\end{lemma}
\subsection{The "generalized subdifferential"}

A main tool used in the present paper is the subdifferential theory of Clark \cite{418,419} for locally Lipschitz functionals. Let $\mathbb{X}$ be a Banach space, $\mathbb{X}^*$ its topological dual, and let $\langle \cdot ,\cdot\rangle_{\mathbb{X}}$ denote the duality brackets for the pair $(\mathbb{X},\mathbb{X}^*)$.
\begin{definition}
Let the functional $\phi:\mathbb{X}\rightarrow \mathbb{R}$. We say that $\phi$ is locally Lipschitz if, for every $x\in \mathbb{X}$, there exists an open neighborhood $U(x)$ of $x$ and $k_x>0$ such that
$$\vert \phi(u)-\phi(v)\vert\leq k_x \Vert u-v\Vert_{\mathbb{X}}\ \ \text{for all}\ u,v\in U(x).$$
\end{definition}
\begin{definition}
Given a locally Lipschitz function $\phi:\mathbb{X}\rightarrow \mathbb{R}$, the "generalized directional derivative" of $\phi$ at $u\in\mathbb{X}$ in the direction $v\in \mathbb{X}$, denoted by $\tilde{\phi}(u;v)$, is defined by 
$$\tilde{\phi}(u;v)=\mathop{\limsup_{x\rightarrow u}}_{t\searrow 0}\frac{\phi(x+tv)-\phi(x)}{t}.$$
\end{definition}
\begin{definition}
The "generalized subdifferential" of $\phi$ at $u\in\mathbb{X}$ is the set $\partial\phi(u)\subseteq \mathbb{X}^*$ given by
$$\partial\phi(u)=\lbrace \phi^*\in \mathbb{X}^*:\langle \phi^*,v\rangle_{\mathbb{X}}\leq \tilde{\phi}(u;v)\ \ \text{for all}\ v\in \mathbb{X}\rbrace.$$
\end{definition}
The Hahn-Banach theorem implies that $\partial \phi (u)\neq \emptyset$ for all $u\in\mathbb{X}$, it is convex and $w^*$-compact (in  weak topology sense). If $\phi$ is also convex, then it coincides with the subdifferential in the sense of convex functionals (see \cite{420}). If $\phi\in C^{1}(\mathbb{X},\mathbb{R})$, then $\partial \phi (u)=\lbrace \phi ^{'}(u)\rbrace$. Note that the generalized subdifferential has a remarkable calculus, similar to that in the classical derivative (see \cite{418,419,420}).


In the following section, we state our hypotheses (on $f$ and $K$) and main result.

\section{Hypotheses on $f$ and $K$, and main result}
Next, we set  the hypotheses on the weight function $K(\cdot)$ and the reaction function $f(\cdot,\cdot)$.

\begin{enumerate}
\item[$(H_K)$] $K:\ \mathbb{R}^{d}\longrightarrow\mathbb{R}$ is a continuous function and satisfies
\begin{enumerate}
\item[$(K_1)$]
 $K(x)>0$, for all $x\in \mathbb{R}^{d}$ and  $K\in L^{\infty}(\mathbb{R}^{d})$.
 \item[$(K_2)$]
If\ $\lbrace A_{n}\rbrace_{n\in \mathbb{N}} \subset \mathbb{R}^{d}$ is a sequence of Borel sets such that the Lebesgue measure $mes(A_{n})\leq R$, for all $n \in \mathbb{N}$ and some $R>0$, then
$$\displaystyle{\lim_{r\rightarrow +\infty}\int_{A_{n}\cap B_{r}^{c}(0)}K(x)\ dx =0}, \ \ \text{uniformly in} \  n \in \mathbb{N}.$$
\end{enumerate}
\end{enumerate}
\begin{enumerate}
\item[$(H_f)$] $f:\ \mathbb{R}^d\times\mathbb{R}\longrightarrow\mathbb{R}$ is a measurable function such that for a.a. $x\in \mathbb{R}^d$, $f(x,0)=0$, $f(x,.)$ is locally Lipschitz  and
\begin{enumerate}
\item[$(f_1)$]
$\displaystyle{\lim_{\vert s\vert \rightarrow 0^{+}}\frac{f(x,s)}{g(s)}=0}$ uniformly in $x\in \mathbb{R}^d$.
\item[$(f_2)$]
$\displaystyle{\lim_{\vert s\vert \rightarrow +\infty }\frac{f(x,s)}{g_{*}(s)}=0}$ uniformly in $x\in \mathbb{R}^d$.
\item[$(f_3)$]
$\displaystyle{\lim\limits_{ s\rightarrow \pm\infty}\frac{F(x,s)}{\vert s\vert^{g^+}}}=+\infty$ uniformly in $x\in \mathbb{R}^d$, where $\displaystyle{F(x,s)=\int_0^s f(x,t)dt}$.
\item[$(f_4)$]$0< (g^+-1)f(x,s)s< f^{*}(x,s)s^2$, for a.a. $  x\in \mathbb{R}^d$, all $f^*(x,s)\in\partial_s f(x,s)$, and all $\vert s\vert>0$.
\end{enumerate}
\end{enumerate}
\begin{rem}\label{4rem1}
$\bullet$ Under hypothesis $(f_4)$ and by the generalized subdifferential calculus of Clarke \cite[ p. 48]{419}, for a.a. $x\in\mathbb{R}^d$, we have
\begin{equation}\label{4eq50}
s\mapsto \frac{f(x,s)}{\vert s\vert^{g^+-1}}\ \ \text{is increasing on}\ (0,+\infty)\ \text{and on}\ (-\infty,0)
\end{equation}
and
\begin{equation}\label{4eq51}
s\mapsto f(x,s)s-g^+F(x,s)\ \ \text{is increasing on}\ [0,+\infty)\ \text{and decreasing on}\ (-\infty,0]. 
\end{equation}
$\bullet$ The assumption $(H_f)$ is weaker than the (AR)  condition. Indeed, the  function  $\displaystyle{f(s)=\vert s\vert^{g^+-2}s\ln(1+\vert s\vert)}$ $($for the sake of simplicity, we drop the $x$-dependence$)$ satisfies hypotheses $(H_f)$ but not the (AR) condition. \\
$\bullet$ By hypothesis $(f_4)$ and the fact that $f(x,0)=0$, for a.a. $x\in \mathbb{R}^d$, we have 
  $$f(x,s)\geq 0\ (\leq 0),\  \text{for a.a.}\ x\in \mathbb{R}^d\ \text{and all}\ s\geq 0\ (s\leq 0).$$
  Therefore, 
  $$F(x,t)=\int_0^t f(x,s)ds \geq 0,\ \text{for a.a.}\ x\in \mathbb{R}^d\ \text{and all}\ t\geq 0.$$
  On the other hand, if $t<0$, by \eqref{4eq50}, for a.a. $x\in \mathbb{R}^d$, we have  
\begin{align*}
  F(x,t)&=\int_0^t f(x,s)ds=\int_0^t \frac{f(x,s)}{\vert s\vert^{g^+-1}}\vert s\vert^{g^+-1} ds\\
  &\geq \frac{f(x,t)}{\vert t\vert^{g^+-1}}\int_0^t \vert s\vert^{g^+-1} ds=\frac{1}{g^+}f(x,t)t\\
  & \geq 0\ \ \ \ (\text{since}\ t\leq s<0\ \text{and}\  f(t)<0).
\end{align*}
Thus,
\begin{equation}\label{4eqhlel}
F(x,t)\geq 0,\ \text{for a.a.}\ x\in \mathbb{R}^d\ \text{and all}\ t\in \mathbb{R}.
\end{equation} 
\end{rem}

Now, we state our main result for problem \eqref{P}.

\begin{thm}\label{4thm1}
Suppose that the hypotheses $(H_f)$, $(H_G)$ and $(H_K)$ hold, then problem \eqref{P} has a ground state solution $\widehat{u}\in W^{\alpha,G}(\mathbb{R}^d)$ of fixed sign and a nodal weak solution $\widehat{w}\in W^{\alpha,G}(\mathbb{R}^d)$.
\end{thm}

In the following section, we give the energy functional associated to problem \eqref{P} and some technical lemmas.

\section{Energy functional and technical lemmas}
We start by given the energy functional associated to problem \eqref{P}.

Let $J:\ W^{\alpha,G}(\mathbb{R}^{d})\rightarrow \mathbb{R}$ be the energy functional associated to problem \eqref{P} defined by
$$J(u)=\int_{\mathbb{R}^{d}}\int_{\mathbb{R}^{d}}G\left( \frac{u(x)-u(y)}{\vert x-y\vert^{\alpha}}\right)\frac{dxdy}{\vert x-y\vert^{d}} + \int_{\mathbb{R}^{d}}G(u)\ dx-\int_{\mathbb{R}^{d}}K(x)F(x,u)\ dx, \ \ \text{for all}\ u\in W^{\alpha,G}(\mathbb{R}^{d}).$$
In view of hypotheses $(H_G)$, $(H_f)$ and $(H_K)$,  $J\in C^1(W^{\alpha,G}(\mathbb{R}^{d}))$  and
\begin{align*}
\langle J^{'}(u),v\rangle & =\int_{\mathbb{R}^{d}}\int_{\mathbb{R}^{d}}g\left( \frac{u(x)-u(y)}{\vert x-y\vert^{\alpha}}\right)\frac{v(x)-v(y)}{\vert x-y\vert^{\alpha+d}}dxdy+\int_{\mathbb{R}^{d}}g(u)v\ dx
 -\int_{\mathbb{R}^{d}}K(x)f(x,u)v\ dx,
\end{align*}
for all $u,v \in W^{\alpha,G}(\mathbb{R}^{d})$.
\begin{definition}
We say that $u\in W^{\alpha,G}(\Omega)$ is a weak solution of problem \eqref{P} if 
\begin{align*}
\int_{\mathbb{R}^{d}}\int_{\mathbb{R}^{d}}g\left( \frac{u(x)-u(y)}{\vert x-y\vert^{\alpha}}\right)\frac{v(x)-v(y)}{\vert x-y\vert^{\alpha+d}}dxdy+\int_{\mathbb{R}^{d}}g(u)v\ dx
 =\int_{\mathbb{R}^{d}}K(x)f(x,u)v\ dx,\ \text{for all}\ W^{\alpha,G}(\mathbb{R}^{d}).
\end{align*}
\end{definition}\ \\

Let $\mathcal{N}$ be the Nehari manifold for $J$, defined by
$$\mathcal{N}=\left\lbrace u\in W^{\alpha,G}(\mathbb{R}^d):u\neq 0,\ \langle J^{'}(u),u\rangle=0\right\rbrace $$
where $\langle.,.\rangle$ is the duality brackets for the pair $\left( (W^{\alpha,G}(\mathbb{R}^d))^*,W^{\alpha,G}(\mathbb{R}^d)\right)$. Since we seek nodal solutions, we consider the following subset of $\mathcal{N}$  
$$\mathcal{M}=\left\lbrace w\in W^{\alpha,G}(\mathbb{R}^d):\ w^+\neq 0, w^-\neq 0,\langle J^{'}(w),w^+\rangle=\langle J^{'}(w),w^-\rangle=0\right\rbrace.$$
Recall that $w^+=\max\left\lbrace  w,0\right\rbrace $, $w^-=\min\left\lbrace  w,0\right\rbrace $ for  $w\in W^{\alpha,G}(\mathbb{R}^d)$. Evidently,  $w^\pm\in W^{\alpha,G}(\mathbb{R}^d)$, $w=w^++w^-$ and $\vert w\vert=\vert w^+\vert+\vert w^-\vert$.\\

In what follows, we give some technical lemmas and results which are crucial for the proof of Theorem \ref{4thm1}.

\begin{lemma}\label{4lem3}
Assume that hypothesis $(g_3)$ holds. Then, the function defined by 
$$s\mapsto G(s)-\frac{1}{g^+}g(s)s$$ is non-decreasing on $(0,+\infty)$ and non-increasing on $(-\infty,0)$.
\end{lemma}
\begin{proof}
From $(g_3)$, we have,
$$g(s)-\frac{1}{g^+}g^{'}(s)s-\frac{1}{g^+}g(s)\geq \left( \frac{g^+-1}{g^+}-\frac{g^--1}{g^+}\right) g(s)\geq 0,\ \ \text{for all}\ s\in (0,+\infty).$$
Which gives that,
$$s\mapsto G(s)-\frac{1}{g^+}g(s)s\ \text{is non-decreasing on}\ (0,+\infty).$$
Exploiting the fact that $ G(s)-\frac{1}{g^+}g(s)s$ is even on $\mathbb{R}$, we deduce that
$$s\mapsto G(s)-\frac{1}{g^+}g(s)s\ \text{is non-increasing on}\ (-\infty,0).$$
This ends the proof.
\end{proof}
\begin{lemma}\label{4lem4}\cite[Lemma 3.2]{47}\\
Assume that the hypotheses $(f_1)-(f_2)$ and $(H_K)$ hold. Let $\lbrace u_n\rbrace_{n\in\mathbb{N}}$ be a sequence such that $u_n\rightharpoonup u$ in $W^{\alpha,G}(\mathbb{R}^d)$, then
\begin{enumerate}
\item[$(1)$] $\displaystyle{\lim_{n\rightarrow +\infty}\int_{\mathbb{R}^d}K(x)F(x,u_n) dx=\int_{\mathbb{R}^d}K(x)F(x,u) dx}.$
\item[$(2)$] $\displaystyle{\lim_{n\rightarrow +\infty}\int_{\mathbb{R}^d}K(x)f(x,u_n)u_n dx=\int_{\mathbb{R}^d}K(x)f(x,u)u dx}.$
\item[$(3)$] $\displaystyle{\lim_{n\rightarrow +\infty}\int_{\mathbb{R}^d}K(x)f(x,u_n^\pm)u_n^\pm dx=\int_{\mathbb{R}^d}K(x)f(x,u^\pm)u^\pm dx}.$
\end{enumerate}
\end{lemma}
\begin{lemma}\label{4lem5}\cite[Lemma 4.3]{47}\\
Assume that the hypotheses $(H_f)$, $(H_G)$ and $(H_K)$ hold. Let $w\in \mathcal{M}$ , then 
$$\langle J^{'}(w^\pm),w^\pm\rangle \leq \langle J^{'}(w),w^\pm\rangle.$$
\end{lemma}
\begin{lemma}\label{4lem9}
Assume that the hypotheses $(H_f)$, $(H_G)$ and $(H_K)$ hold. Then, for all $w\in W^{\alpha,G}(\mathbb{R}^d)$  such that $w^\pm\neq 0$, we have 
$$J(w^+)+J(w^-)< J(w).$$
\end{lemma}
\begin{proof}
Let $w\in W^{\alpha,G}(\mathbb{R}^d)$  such that $w^\pm\neq 0$. We argue by contradiction, suppose that 
\begin{equation}\label{4eq402}
J(w)\leq J(w^+)+J(w^-).
\end{equation}
Firstly, we observe that 
\begin{align*}
J(w) &=J(w^+)+J(w^-)\\
& +\int_{\text{supp}(w^-)}\int_{\text{supp}(w^+)}G\left( \frac{w^+(x)-w^-(y)}{\vert x-y\vert^{\alpha}}\right)\frac{dxdy}{\vert x-y\vert^{d}} \\
&+ \int_{\text{supp}(w^+)}\int_{\text{supp}(w^-)}G\left( \frac{w^-(x)-w^+(y)}{\vert x-y\vert^{\alpha}}\right)\frac{dxdy}{\vert x-y\vert^{d}}. 
\end{align*}
By \eqref{4eq402} and the fact that $G(\cdot)$ is an  even positive function, then
\begin{equation}\label{4eq403}
\left\lbrace 
\begin{array}{ll}
\displaystyle{\int_{\text{supp}(w^-)}\int_{\text{supp}(w^+)}G\left( \frac{w^+(x)-w^-(y)}{\vert x-y\vert^{\alpha}}\right)\frac{dxdy}{\vert x-y\vert^{d}}= 0,} \\
\ \\
\displaystyle{\int_{\text{supp}(w^+)}\int_{\text{supp}(w^-)}G\left( \frac{w^-(x)-w^+(y)}{\vert x-y\vert^{\alpha}}\right)\frac{dxdy}{\vert x-y\vert^{d}}= 0.}
\end{array}
\right.
\end{equation}
Thus, one has 
$$w^+(x)=w^-(y)\ \text{and}\ w^-(x)=w^+(y),\ \text{for a.a.}\ \ x,y\in\mathbb{R}^d.$$
Therefore, $w^\pm=0$ in $\mathbb{R}^d$ ( since supp$(w^+)\displaystyle{\cap}$supp$(w^-)=\emptyset$ ). Which is a contradiction.\\
This ends the proof.
\end{proof}
\section{Ground state solutions}
In this section, we prove the existence of a ground state solution for problem \eqref{P}.

\begin{prop}\label{4prop2}
Assume that the hypotheses $(H_f)$, $(H_G)$ and $(H_K)$ hold, then for all $u\in W^{\alpha,G}(\mathbb{R}^d)\backslash \lbrace 0\rbrace$, there exists a unique $t_{u}>0$ such that $t_{u}u\in \mathcal{N}$.
\end{prop}
\begin{proof}
Let $u\in W^{\alpha,G}(\mathbb{R}^d)\backslash \lbrace 0\rbrace$, consider the following fibering map $h_u:(0,+\infty)\longrightarrow \mathbb{R}$ defined by 
$$h_u(t)=\langle J^{'}(tu),tu\rangle\ \ \text{for all}\ t>0.$$
From hypotheses $(f_1)$ and $(f_2)$, for all $\varepsilon>0$ there exists $C_\varepsilon>0$ such that
\begin{equation}\label{4eq81}
f(x,s)s\leq \varepsilon g(s)s+C_\varepsilon g_*(s)s,\ \ \text{for a.a.}\ x\in \mathbb{R}^d \ \text{and all}\ s\in \mathbb{R}.
\end{equation}
Using \eqref{4eq81}, Lemma \ref{4lem1}, and assumptions $(g_2)$ and $(K_1)$, we get 

\begin{align}\label{4eq7546}
h_{u}(t) & \geq g^-t^{g^+}\int_{\mathbb{R}^{d}}\int_{\mathbb{R}^{d}}G\left( \frac{u(x)-u(y)}{\vert x-y\vert^{\alpha}}\right)\frac{dxdy}{\vert x-y\vert^{d}} + \int_{\mathbb{R}^{d}}g(tu)tudx \nonumber\\ 
& -\varepsilon\Vert K\Vert_{\infty}\int_{\mathbb{R}^{d}}g(tu)tu dx-C_\varepsilon\Vert K\Vert_{\infty}\int_{\mathbb{R}^{d}}g_*(tu)tu dx\nonumber\\
& \geq g^-t^{g^+}\int_{\mathbb{R}^{d}}\int_{\mathbb{R}^{d}}G\left( \frac{u(x)-u(y)}{\vert x-y\vert^{\alpha}}\right)\frac{dxdy}{\vert x-y\vert^{d}} + \left( 1-\varepsilon\Vert K\Vert_{\infty}\right)g^- t^{g^+} \int_{\mathbb{R}^{d}}G(u)dx\nonumber \\
& -g_*^+ t^{g^-_*}\int_{\mathbb{R}^{d}}G_*(u) dx\nonumber\\
& =g^-t^{g^+}\left[ \int_{\mathbb{R}^{d}}\int_{\mathbb{R}^{d}}G\left( \frac{u(x)-u(y)}{\vert x-y\vert^{\alpha}}\right)\frac{dxdy}{\vert x-y\vert^{d}} + \left( 1-\varepsilon\Vert K\Vert_{\infty}\right)\int_{\mathbb{R}^{d}}G(u)dx\right]\nonumber \\
& -g_*^+ t^{g^-_*}\int_{\mathbb{R}^{d}}G_*(u) dx.
\end{align}
Choosing $\displaystyle{\varepsilon \in (0,\frac{1}{\Vert K\Vert_{\infty}})}$ in \eqref{4eq7546} and using the fact that $g^+< g^-_*,$ we find  $t_0>0$ small enough such that 
$$h_{u}(t)> 0,\ \forall\ t\in (0,t_0).$$
Let $A\subset \text{supp}(u)$ such that the Lebesgue  measure of $A$ is positive, $mes (A)>0$.\\
 Using \eqref{4eq51}, \eqref{4eqhlel}, Lemmas \ref{4lem1}, \ref{4lem2}, assumptions  $(g_2)$, and $(K_1)$ , for $t$ large, we find
\begin{align}\label{4eq4}
\frac{h_{u}(t)}{t^{g^+}} & \leq g^+\max\left\lbrace \Vert u\Vert^{g^-},\Vert u\Vert^{g^+}\right\rbrace - \int_{\mathbb{R}^d}K(x)\frac{f(x,tu)tu}{ t^{g^+}}dx\nonumber\\
& \leq g^+\left[ \max\left\lbrace \Vert u\Vert^{g^-},\Vert u\Vert^{g^+}\right\rbrace - \int_{\mathbb{R}^d}K(x)\frac{F(x,tu)}{t^{g^+}}dx\right]\nonumber\\
&\leq g^+\left[ \max\left\lbrace \Vert u\Vert^{g^-},\Vert u\Vert^{g^+}\right\rbrace - \int_{A}K(x)\frac{F(x,tu)}{ \vert tu\vert^{g^+}}\vert u\vert^{g^+}dx\right] 
\end{align}
In light of hypothesis $(f_3)$, we see that 
\begin{equation}\label{4eq5}
\lim\limits_{t\rightarrow +\infty}\frac{F(x,tu)}{\vert tu\vert^{g^+}}\vert u\vert^{g^+}=+\infty,\ \text{uniformly for all}\ x\in A.
\end{equation}
By assumption $(K_1)$ and Fatou's lemma, it follows that
\begin{equation}\label{4eq3}
\int_{A}K(x)\frac{F(x,tu)}{ \vert tu\vert^{g^+}}\vert u\vert^{g^+}dx \rightarrow +\infty\ \text{as}\ t\rightarrow +\infty.
\end{equation}
From \eqref{4eq4} and \eqref{4eq3}, we get
$$\limsup\limits_{t\rightarrow +\infty}\frac{h_u(t)}{t^{g^+}}\leq -\infty.$$
Therefore, there is $t_1>0$ large enough such that 
$$h_u(t)<0,\ \text{for all}\ t\in (t_1,+\infty).$$
According to Bolzano's theorem, there exists $t_u>0$ such that $h_u(t_u)=0$. Hence, $\langle J^{'}(t_{u}u),t_{u}u\rangle =0$, that is, $t_{u}u\in \mathcal{N}.$\\

In what follows, we prove the uniqueness of $t_u>0$. Let $t_1,t_2$ be the two different positive number such that  $t_{i}u\in \mathcal{N},\ i=1,2$.\\ 

Firstly, we consider the case $u\in \mathcal{N}$. Without loss of generality, we may take $t_1=1$ and $t_1\neq t_2$. Thus,
\begin{equation}\label{4eq100}
\langle J^{'}(u),u\rangle =0
\end{equation}
and
\begin{equation}\label{4eq101}
\langle J^{'}(t_2u),u\rangle =0.
\end{equation}
From \eqref{4eq100}, it yields that
\begin{align}\label{4eq102}
\int_{\mathbb{R}^{d}}K(x)\frac{f(x,u)}{\vert u\vert ^{g^+-1}}\vert u\vert ^{g^+-1} udx&=\int_{\mathbb{R}^{d}}\int_{\mathbb{R}^{d}}g\left( \frac{u(x)-u(y)}{\vert x-y\vert^{\alpha}}\right)\frac{u(x)-u(y)}{\vert x-y\vert^{\alpha +d}}dxdy\nonumber\\ &+\int_{\mathbb{R}^{d}}g(u)udx.
\end{align}
If $t_2<t_1=1$, then, by \eqref{4eq101} and Lemma \ref{4lem1}, we obtain
\begin{align}\label{4eq103}
\int_{\mathbb{R}^{d}}K(x)\frac{f(x,t_2u)}{\vert t_2 u\vert ^{g^+-1}}\vert u\vert ^{g^+-1} u dx & \geq\int_{\mathbb{R}^{d}}\int_{\mathbb{R}^{d}}g\left( \frac{u(x)-u(y)}{\vert x-y\vert^{\alpha}}\right)\frac{u(x)-u(y)}{\vert x-y\vert^{\alpha +d}}dxdy\nonumber\\
 &+\int_{\mathbb{R}^{d}}g(u)udx.
\end{align}
Subtracting \eqref{4eq103} from \eqref{4eq102}, using  \eqref{4eq50} and hypothesis $(K_1)$, we infer that
\begin{align}\label{4eq106}
 0< \int_{\mathbb{R}^{d}}K(x)\vert u\vert^{g^+-1}u\left[ \frac{f(x,u)}{\vert u\vert^{g^+-1}}- \frac{f(x,t_2u)}{\vert t_2 u\vert^{g^+-1}}\right] \ dx\leq 0.
\end{align}
Which is a contradiction.\\

If $t_2>t_1=1$, then, by \eqref{4eq101} and Lemma \ref{4lem1}, we get
\begin{align}\label{4eq104}
\int_{\mathbb{R}^{d}}K(x)\frac{f(x,t_2u)}{\vert t_2 u\vert ^{g^+-1}}\vert u\vert ^{g^+-1} udx& \leq\int_{\mathbb{R}^{d}}\int_{\mathbb{R}^{d}}g\left( \frac{u(x)-u(y)}{\vert x-y\vert^{\alpha}}\right)\frac{u(x)-u(y)}{\vert x-y\vert^{\alpha +d}}dxdy\nonumber\\
 &+\int_{\mathbb{R}^{d}}g(u)udx.
\end{align}
Subtracting \eqref{4eq104} from \eqref{4eq100}, using  \eqref{4eq50} and hypothesis $(K_1)$, we deduce that
\begin{align}\label{4eq107}
 0\leq \int_{\mathbb{R}^{d}}K(x)\vert u\vert^{g^+-1}u\left[ \frac{f(x,u)}{\vert u\vert^{g^+-1}}- \frac{f(x,t_2u)}{\vert t_2 u\vert^{g^+-1}}\right] \ dx<0.
\end{align}
Which is a contradiction too. Therefore $t_1=t_2=1$.\\

Secondly, for the case $u\notin \mathcal{N}$. Let $v=t_1u\in \mathcal{N}$, $t_1\neq 1$ and $\displaystyle{t_2u=\frac{t_2}{t_1} t_1u=\frac{t_2}{t_1}v\in \mathcal{N}}$.
Applying the same arguments  as above, we prove that $\displaystyle{\frac{t_2}{t_1}=1}$.\\
This ends the proof.
\end{proof}
\begin{prop}\label{4prop3}
Assume that the hypotheses $(H_f)$, $(H_G)$ and $(H_K)$ hold. Then for all $u\in\mathcal{N}$,  
$$J(tu)\leq J(u),\  \text{for all}\  t>0.$$
\end{prop}
\begin{proof}
Let $u\in\mathcal{N}$ and consider the fibering map $k_u:(0,+\infty)\longrightarrow \mathbb{R}$ defined by 
$$k_u(t)=J(tu)\ \ \text{for all}\ t>0.$$
By the hypotheses $(f_1)$ and $(f_2)$, for all $\varepsilon>0$ there exists $C_\varepsilon >0$ such that 
\begin{equation}\label{4eq1}
F(x,s)\leq \varepsilon G(s)+C_\varepsilon G_*(s),\ \text{for a.a.}\ x\in\mathbb{R}^d\ \text{ and all}\ \vert s\vert >0.
\end{equation}
Using \eqref{4eq1}, Theorem \ref{4thm2}, Lemmas \ref{4lem1}, \ref{4lem2} and hypothesis $(K_1)$, then
\begin{align*}
k_u(t)&=\int_{\mathbb{R}^{d}}\int_{\mathbb{R}^{d}}G\left( \frac{tu(x)-tu(y)}{\vert x-y\vert^{\alpha}}\right)\frac{dxdy}{\vert x-y\vert^{d}} + \int_{\mathbb{R}^{d}}G(tu)\ dx-\int_{\mathbb{R}^{d}}K(x)F(x,tu)\ dx\\
&\geq t^{g^+}\int_{\mathbb{R}^{d}}\int_{\mathbb{R}^{d}}G\left( \frac{u(x)-u(y)}{\vert x-y\vert^{\alpha}}\right)\frac{dxdy}{\vert x-y\vert^{d}} + \left( 1-\varepsilon\Vert K\Vert_{\infty}\right) \int_{\mathbb{R}^{d}}G(tu)\ dx\\
&-C_\varepsilon\Vert K\Vert_{\infty}\int_{\mathbb{R}^{d}}G_*(tu)\ dx\\
& \geq t^{g^+}\left( 1-\varepsilon\Vert K\Vert_{\infty}\right)\rho (\alpha; u)-t^{g^-_*}C_\varepsilon \Vert K\Vert_{\infty}\int_{\mathbb{R}^{d}}G_*(u)\ dx.
 \end{align*} 
Taking $\displaystyle{\varepsilon \in (0,\frac{1}{\Vert K\Vert_{\infty}})}$, having in mind  that $g^+<g^-_*$, there is  $t_0 > 0$ sufficiently small such that
\begin{equation}\label{4eq2}
0<k_u(t),\ \text{for all}\ t\in (0,t_0).
\end{equation}
Let $A\subset \text{supp}(u)$ such that the Lebesgue  measure of $A$ is positive, $mes(A)>0$.\\
It's clear that, for $t$ large enough,
\begin{align}\label{4eq7}
\frac{k_{u}(t)}{t^{g^+}}\leq g^+\max\left\lbrace \Vert u\Vert^{g^-},\Vert u\Vert^{g^+}\right\rbrace - \int_{A}K(x)\frac{F(x,tu)}{ \vert tu\vert^{g^+}}\vert u\vert^{g^+}dx.
\end{align}
Under \eqref{4eq3},
 $$\limsup\limits_{t\rightarrow +\infty} k_u(t)\leq -\infty.$$
 Therefore, the map $k_u(.)$ has  a global maximum $t_u>0$. So, $t_u$ is a critical point for $k_u(.)$ :
 
$$\langle J^{'}(t_u u),u\rangle=0, t_u u\in \mathcal{N}.$$
 By Proposition \ref{4prop2} and the fact that $u\in \mathcal{N}$, we deduce that $t_u=1$.\\ Hence, 
$$J(tu)\leq J(u),\ \ \text{for all}\ t>0.$$
This completes the proof.
\end{proof}
In order to prove the existence of a ground state solution for problem \eqref{P}, we consider the following minimization problem 
$$ m_0:=\inf\limits_{\mathcal{N}}J.$$
 
\begin{prop}\label{4prop44}
Suppose that the hypotheses $(H_f)$, $(H_G)$ and $(H_K)$ are satisfied, then $0<m_0$.
\end{prop}
\begin{proof}
Let $u\in W^{\alpha,G}(\mathbb{R}^d)\backslash\lbrace 0\rbrace$ such that $\Vert u\Vert\leq 1$. Using assumption $(K_1)$, \eqref{4eq1}, Lemma \ref{4lem1} and Theorem \ref{4thm2}, we get
\begin{align*}
J(u)&=\int_{\mathbb{R}^{d}}\int_{\mathbb{R}^{d}}G\left( \frac{u(x)-u(y)}{\vert x-y\vert^{\alpha}}\right)\frac{dxdy}{\vert x-y\vert^{d}} + \int_{\mathbb{R}^{d}}G(u)\ dx-\int_{\mathbb{R}^{d}}K(x)F(x,u)\ dx\\
&\geq \int_{\mathbb{R}^{d}}\int_{\mathbb{R}^{d}}G\left( \frac{u(x)-u(y)}{\vert x-y\vert^{\alpha}}\right)\frac{dxdy}{\vert x-y\vert^{d}} + \left( 1-\varepsilon\Vert K\Vert_{\infty}\right) \int_{\mathbb{R}^{d}}G(u)\ dx\\
&-C_\varepsilon\Vert K\Vert_{\infty}\int_{\mathbb{R}^{d}}G_*(u)\ dx\\
& \geq \left( 1-\varepsilon\Vert K\Vert_{\infty}\right)\rho (\alpha;u)-C_\varepsilon \Vert K\Vert_{\infty}\int_{\mathbb{R}^{d}}G_*(u)\ dx\\
& \geq \left( 1-\varepsilon\Vert K\Vert_{\infty}\right)\Vert u\Vert ^{g^+} -C_\varepsilon \Vert K\Vert_{\infty}\max\left\lbrace \Vert u\Vert_{(G_*)}^{g^-_*},\Vert u\Vert_{(G_*)}^{g^+_*}\right\rbrace\\
 & \geq \left( 1-\varepsilon\Vert K\Vert_{\infty}\right)\Vert u\Vert ^{g^+} -C_\varepsilon \Vert K\Vert_{\infty}C_1\Vert u\Vert^{g^-_*},\ \text{for all}\ \varepsilon>0.
 \end{align*} 
Where $C_\varepsilon$ is the constant given by \eqref{4eq1}. Taking $\displaystyle{\varepsilon \in (0,\frac{1}{\Vert K\Vert_{\infty}})}$ and using the fact that $g^+<g^-_*$, we find $\varrho \in (0,1)$  small enough and $\eta>0$ such that
$$J(u)\geq \eta, \ \text{for all}\ \Vert u\Vert=\varrho.$$
Let $u\in \mathcal{N}$, choosing $t_u>0$ such that $\Vert t_u  u\Vert=\varrho$. By  Proposition \ref{4prop3}, $J(u)\geq J(t_uu)\geq\eta>0.$\\ Therefore, $m_0>0$. Thus the proof.
\end{proof}

The infimum of $J$ is attained on $\mathcal{N}$:

\begin{prop}\label{4prop5}
Under hypotheses $(H_f)$, $(H_G)$ and $(H_K)$, there exists $\widehat{u}\in\mathcal{N}$ such that $J(\widehat{u})=m_0$.
\end{prop}
\begin{proof}
Let $\lbrace u_n\rbrace_{n\in\mathbb{N}}\subset \mathcal{N}$ such that 
\begin{equation}\label{4eq85}
J(u_n)\mathop{\rightarrow}_{n\rightarrow +\infty} m_0.
\end{equation}
Firstly, we prove that $\lbrace u_n\rbrace_{n\in\mathbb{N}}$ is bounded in $W^{\alpha,G}(\mathbb{R}^N)$. We argue by contradiction, assume that there exists a subsequence, denoted again by $\lbrace u_n\rbrace_{n\in\mathbb{N}}$ such that 
$$\Vert u_n\Vert \rightarrow +\infty\ \text{as}\ n\rightarrow +\infty,$$
let 
\begin{equation}\label{4eq300}
v_n:=\frac{u_n}{\Vert u_n\Vert}\ \text{for all}\ n\in\mathbb{N}.
\end{equation}
Since $\lbrace v_n\rbrace_{n\in\mathbb{N}}$ is bounded in $W^{\alpha,G}(\mathbb{R}^d)$, which is a reflexive space, there exists $v\in W^{\alpha,G}(\mathbb{R}^d)$ such that 
$$v_n\rightharpoonup v \ \text{in}\ W^{\alpha,G}(\mathbb{R}^d),$$
and
\begin{equation}\label{4eq55}
v_n(x)\rightarrow v(x)\ \text{as}\ n\rightarrow+\infty,\ \text{for a.a. in}\ \mathbb{R}^{d}.
\end{equation}
Let prove that $v\neq 0$. Since $\lbrace u_n\rbrace_{n\in\mathbb{N}} \in \mathcal{N}$, according to Proposition \ref{4prop3} and Lemma \ref{4lem2}, for all $t\geq 0$, we have  
\begin{align}\label{4eq9}
J(u_n) & = J(\Vert u_n\Vert v_n )\geq J(t v_n)\nonumber\\
& = \int_{\mathbb{R}^{d}}\int_{\mathbb{R}^{d}}G\left( \frac{tv_n(x)-tv_n(y)}{\vert x-y\vert^{\alpha}}\right)\frac{dxdy}{\vert x-y\vert^{d}} + \int_{\mathbb{R}^{d}}G(tv_n)dx\nonumber \\ 
& -\int_{\mathbb{R}^{d}}K(x)F(x,tv_n) dx\nonumber\\
& \geq \min\left\lbrace \Vert tv_n\Vert^{g^-},\Vert tv_n\Vert^{g^+} \right\rbrace  -\int_{\mathbb{R}^{d}}K(x)F(x,tv_n) dx\nonumber\\
& = \min\left\lbrace  t^{g^-},t^{g^+} \right\rbrace  -\int_{\mathbb{R}^{d}}K(x)F(x,tv_n) dx.
\end{align}
Assume that $v_n\rightharpoonup v=0$,  by Lemma \ref{4lem4}, we obtain
$$\int_{\mathbb{R}^{d}}K(x)F(x,tv_n) dx\rightarrow 0,\ \text{for all}\ t>0.$$
 Passing to the limit in \eqref{4eq9}, we get 
$$+\infty>m_0\geq \min\left\lbrace  t^{g^-},t^{g^+} \right\rbrace,\ \text{for all}\ t>0.$$ 
Thus the contradiction, therefore, $v \neq 0$.

Using \eqref{4eq300} and Lemma \ref{4lem2}, we see that
\begin{align*}
J( u_n) & =J(\Vert u_n\Vert v_n )\\
 & \leq g^+\max\left\lbrace \Vert u_n\Vert^{g^-}\Vert v_n\Vert^{g^-},\Vert u_n\Vert^{g^+}\Vert v_n\Vert^{g^+} \right\rbrace - \int_{\mathbb{R}^{d}}K(x)F(x,\Vert u_n\Vert v_n) dx\\
 & \leq g^+\max\left\lbrace \Vert u_n\Vert^{g^-},\Vert u_n\Vert^{g^+} \right\rbrace - \int_{\mathbb{R}^{d}}K(x)F(x,\Vert u_n\Vert v_n) dx,
\end{align*}
which is equivalent to 
\begin{align}\label{eq6}
\frac{J( u_n)}{\max\left\lbrace \Vert u_n\Vert^{g^-},\Vert u_n\Vert^{g^+} \right\rbrace} & \leq g^+- \int_{\mathbb{R}^{d}}K(x)\frac{F(x,\Vert u_n\Vert v_n)}{\max\left\lbrace \Vert u_n\Vert^{g^-},\Vert u_n\Vert^{g^+} \right\rbrace} dx.
\end{align}
Exploiting hypotheses $(f_3)$, $(K_1)$, Fatou's lemma and the fact that $v\neq 0$, we obtain 
\begin{align}\label{eq7}
\liminf_{n\rightarrow +\infty}\int_{\mathbb{R}^{d}}K(x)\frac{F(x,\Vert u_n\Vert v_n)}{\max\left\lbrace \Vert u_n\Vert^{g^-},\Vert u_n\Vert^{g^+} \right\rbrace} dx & =\liminf_{n\rightarrow +\infty}\int_{\mathbb{R}^{d}}K(x)\frac{F(x,\Vert u_n\Vert v_n)}{\Vert u_n\Vert^{g^+}} dx\nonumber\\
&=\liminf_{n\rightarrow +\infty}\int_{\mathbb{R}^{d}}K(x)\frac{F(x,\Vert u_n\Vert v_n)}{\left( \Vert u_n\Vert \vert v_n\vert\right) ^{g^+}}\vert v_n\vert^{g^+} dx \nonumber\\
&=+\infty.
\end{align}
By $(\ref{eq6})$, it yields that 
$$\frac{J( u_n)}{\max\left\lbrace \Vert u_n\Vert^{g^-},\Vert u_n\Vert^{g^+} \right\rbrace}\rightarrow - \infty\ \text{as}\ \ n\rightarrow +\infty,$$
which leads to a contradiction with   \eqref{4eq85}. Therefore, $\lbrace u_n\rbrace_{n\in\mathbb{N}}$ is bounded in $W^{\alpha,G}(\mathbb{R}^d)$. Up to a subsequence, 
$$u_n\rightharpoonup\widehat{u}\ \text{in}\ W^{\alpha,G}(\mathbb{R}^d),$$
and 
\begin{equation}\label{4eq555}
u_n(x)\rightarrow \widehat{u}(x)\ \text{ as}\ n\rightarrow +\infty,\ \text{for a.a.}\ x\in\mathbb{R}^d.
\end{equation}
Suppose that $\widehat{u} =0$, then 
$$0<m_0\leq J(u_n)\mathop{\longrightarrow}_{n\rightarrow +\infty} J(0)=0,$$
which is a contradiction. Thus, $\widehat{u} \neq 0$. According to Proposition \ref{4prop2}, there is a unique $t_{\widehat{u}}>0$ such that 
\begin{equation}\label{4eq122}
t_{\widehat{u}}\widehat{u}\in \mathcal{N}.
\end{equation}
By \eqref{4eq555}, Proposition \ref{4prop3}, Lemma \ref{4lem4} and Fatou's lemma,  it follows that
\begin{align}\label{4eq108}
m_0=\lim_{n\rightarrow +\infty}J(u_n)& \geq \liminf_{n\rightarrow +\infty} J(t_{\widehat{u}}u_n)\geq J(t_{\widehat{u}}\widehat{u})\nonumber\\
& \geq m_0.
\end{align}
Therefore, $m_0=J(t_{\widehat{u}}\widehat{u})=\inf\limits_{\mathcal{N}}J.$\\
Next, we shall prove that $t_{\widehat{u}}=1$. Since $\lbrace u_n\rbrace_{n\in\mathbb{N}}\in\mathcal{N}$, then
\begin{align*}
\int_{\mathbb{R}^{d}}\int_{\mathbb{R}^{d}}g\left( \frac{u_n(x)-u_n(y)}{\vert x-y\vert^{\alpha}}\right)\frac{u_n(x)-u_n(y)}{\vert x-y\vert^{\alpha+d}}dxdy + \int_{\mathbb{R}^{d}}g(u_n)u_ndx
=\int_{\mathbb{R}^{d}}K(x)f(x,u_n)u_n dx,\ \text{for all}\ n\in\mathbb{N}.
\end{align*}
By \eqref{4eq555}, Lemma \ref{4lem4} and Fatou's lemma, we deduce that 
\begin{align}\label{4eq120}
\int_{\mathbb{R}^{d}}\int_{\mathbb{R}^{d}}g\left( \frac{\widehat{u}(x)-\widehat{u}(y)}{\vert x-y\vert^{\alpha}}\right)\frac{\widehat{u}(x)-\widehat{u}(y)}{\vert x-y\vert^{\alpha+d}}dxdy + \int_{\mathbb{R}^{d}}g(\widehat{u})\widehat{u}dx
\leq\int_{\mathbb{R}^{d}}K(x)f(x,\widehat{u})\widehat{u} dx,
\end{align}
Suppose that $t_{\widehat{u}}>1$.
From \eqref{4eq122} and Lemma \ref{4lem1}, one has
\begin{align}\label{4eq800}
\displaystyle{\int_{\mathbb{R}^{d}}\frac{K(x)f(x,t_{\widehat{u}}\widehat{u})\vert \widehat{u}\vert^{g^+-1}\widehat{u}}{\vert t_{\widehat{u}}\widehat{u}\vert^{g^+-1}}\ dx} & \leq \int_{\mathbb{R}^{d}}\int_{\mathbb{R}^{d}}g\left( \frac{\widehat{u}(x)-\widehat{u}(y)}{\vert x-y\vert^{\alpha}}\right)\frac{\widehat{u}(x)-\widehat{u}(y)}{\vert x-y\vert^{\alpha +d}}dxdy\nonumber\\
& +\displaystyle{ \int_{\mathbb{R}^{d}}g(\widehat{u})\widehat{u}\ dx .}
\end{align}
Putting together \eqref{4eq120} and \eqref{4eq800}, using  \eqref{4eq50} and hypothesis $(K_1)$, we find that
\begin{align}\label{h3}
 0\leq \int_{\mathbb{R}^{d}}K(x)\vert u\vert^{g^+-1}u\left[ \frac{f(x,u)}{\vert u\vert^{g^+-1}}- \frac{f(x,t_uu)}{\vert t_u u\vert^{g^+-1}}\right] \ dx< 0.
\end{align}
Thus the contradiction. Therefore, $0<t_u\leq 1$.\\
Suppose that $t_{\widehat{u}}\neq 1$. Using \eqref{4eq51}, Lemmas \ref{4lem3}, \ref{4lem4} , Fatou's lemma and hypothesis $(K_1)$, we see that 
\begin{align*}
m_0&=J(t_{\widehat{u}}\widehat{u})=J(t_{\widehat{u}}\widehat{u})-\frac{1}{g^+}\langle J^{'}(t_{\widehat{u}}\widehat{u}),t_{\widehat{u}}\widehat{u}\rangle\\
&=\int_{\mathbb{R}^d}\int_{\mathbb{R}^d}\left[ G\left( \frac{t_{\widehat{u}}\widehat{u}(x)-t_{\widehat{u}}\widehat{u}(y)}{\vert x-y\vert^{\alpha}}\right)-\frac{1}{g^+}g\left( \frac{t_{\widehat{u}}\widehat{u}(x)-t_{\widehat{u}}\widehat{u}(y)}{\vert x-y\vert^{\alpha}}\right)\frac{t_{\widehat{u}}\widehat{u}(x)-t_{\widehat{u}}\widehat{u}(y)}{\vert x-y\vert^{\alpha}}\right] \frac{dxdy}{\vert x-y\vert^d}\\
& + \int_{\mathbb{R}^d} G(t_{\widehat{u}}\widehat{u})-\frac{1}{g^+}g(t_{\widehat{u}}\widehat{u})t_{\widehat{u}}\widehat{u} dx+\int_{\mathbb{R}^d} K(x)\left[ \frac{1}{g^+}f(x,t_{\widehat{u}}\widehat{u})t_{\widehat{u}}\widehat{u}-F(x,t_{\widehat{u}}\widehat{u}) \right]dx\\
&<\int_{\mathbb{R}^d}\int_{\mathbb{R}^d}\left[ G\left( \frac{\widehat{u}(x)-\widehat{u}(y)}{\vert x-y\vert^{\alpha}}\right)-\frac{1}{g^+}g\left( \frac{\widehat{u}(x)-\widehat{u}(y)}{\vert x-y\vert^{\alpha}}\right)\frac{\widehat{u}(x)-\widehat{u}(y)}{\vert x-y\vert^{\alpha}}\right] \frac{dxdy}{\vert x-y\vert^d}\\
& + \int_{\mathbb{R}^d} G(\widehat{u})-\frac{1}{g^+}g(\widehat{u})\widehat{u} dx+\int_{\mathbb{R}^d} K(x)\left[ \frac{1}{g^+}f(x,\widehat{u})\widehat{u}-F(x,\widehat{u}) \right]dx \\
&\leq\liminf_{n\rightarrow +\infty}\left[ \int_{\mathbb{R}^d}\int_{\mathbb{R}^d}\left[ G\left( \frac{u_n(x)-u_n(y)}{\vert x-y\vert^{\alpha}}\right)-\frac{1}{g^+}g\left( \frac{u_n(x)-u_n(y)}{\vert x-y\vert^{\alpha}}\right)\frac{u_n(x)-u_n(y)}{\vert x-y\vert^{\alpha}}\right] \frac{dxdy}{\vert x-y\vert^d}\right.\\
& \left.+ \int_{\mathbb{R}^d} G(u_n)-\frac{1}{g^+}g(u_n)u_n dx+\int_{\mathbb{R}^d} K(x)\left[ \frac{1}{g^+}f(x,u_n)u_n-F(x,u_n) \right]dx \right] \\
&=\liminf_{n\rightarrow +\infty} J(u_n)=J(\widehat{u})\\
&=m_0.
\end{align*} 
Which is a contradiction. Thus, $t_{\widehat{u}}=1$. Hence,
 $$m_0=J(\widehat{u})=\inf\limits_{\mathcal{N}}J.$$
This completes the proof.
\end{proof}
In the following we prove that  $\widehat{u}$ is a critical point of the functional $J$.  
\begin{prop}\label{4prop9}
Assume that the hypotheses  $(H_f)$, $(H_G)$ and $(H_K)$ hold. Then, $\widehat{u}$ is a critical point of $J$. Hence, $\widehat{u}$ is a least energy weak solution of problem \eqref{P}.
\end{prop}
\begin{proof}
Let consider the functional $\varphi:W^{\alpha,G}(\mathbb{R}^d)\rightarrow \mathbb{R}$ defined by 
\begin{align*}
\varphi(u) & =\langle J^{'}(u),u\rangle =\int_{\mathbb{R}^{d}}\int_{\mathbb{R}^{d}}g\left( \frac{u(x)-u(y)}{\vert x-y\vert^{\alpha}}\right)\frac{u(x)-u(y)}{\vert x-y\vert^{\alpha}}\frac{dxdy}{\vert x-y\vert^{d}}\\
& + \int_{\mathbb{R}^{d}}g(u)u dx - \int_{\mathbb{R}^{d}} K(x)f(x,u)u dx.
\end{align*}
By hypotheses $(H_f)$, $\varphi$  is locally Lipschitz (see \cite[Theorem 2.7.2, p. 221]{419}).\\
Let $u\in W^{\alpha,G}(\mathbb{R}^d)$,  for all $\varphi^*_u\in \partial \varphi(u)$, there is $f^*(x,u)\in \partial_u f(x,u)$ such that  
\begin{align}\label{4eq500}
\langle\varphi^*_u,v\rangle 
& =\int_{\mathbb{R}^{d}}\int_{\mathbb{R}^{d}}g\left( \frac{u(x)-u(y)}{\vert x-y\vert^{\alpha}}\right)\frac{v(x)-v(y)}{\vert x-y\vert^{\alpha}}\frac{dxdy}{\vert x-y\vert^{d}}\nonumber\\
& + \int_{\mathbb{R}^{d}}g(u)v dx - \int_{\mathbb{R}^{d}} K(x)f(x,u)v dx\nonumber\\
& + \int_{\mathbb{R}^{d}}\int_{\mathbb{R}^{d}}g^{'}\left( \frac{u(x)-u(y)}{\vert x-y\vert^{\alpha}}\right)\frac{[v(x)-v(y)]^{2}}{\vert x-y\vert^{2\alpha}}\frac{dxdy}{\vert x-y\vert^{d}}\nonumber\\
& + \int_{\mathbb{R}^{d}}g^{'}(u)v^{2} dx - \int_{\mathbb{R}^{d}} K(x)f^*(x,u)v^{2} dx,\ \text{for all}\ v\in W^{\alpha,G}(\mathbb{R}^d).
\end{align}
From Proposition \ref{4prop5}, we have 
$$J(\widehat{u})=m_0=\inf\left\lbrace J(u):\ \varphi(u)=0,\ u\in W^{\alpha,G}(\mathbb{R}^d)\backslash\lbrace 0\rbrace \right\rbrace.$$
According to the non-smooth multiplier rule of Clarke \cite[Theorem 10.47, p. 221]{418}, there exists $\lambda_0\geq 0$  such that
$$
0\in\partial(J+ \lambda_0 \varphi)(\widehat{u}).
$$
 By the subdifferential calculus of Clarke \cite[p. 48]{419}, it follows that 

$$0\in\partial J(\widehat{u})+\lambda_0 \partial\varphi(\widehat{u}).$$
Thus,
\begin{equation}\label{4eq11}
0=J^{'}(\widehat{u})+\lambda_0 \varphi^*_{\widehat{u}}\ \ \text{in}\ (W^{\alpha,G}(\mathbb{R}^d))^*,\ \ \text{for all}\ \varphi^*_{\widehat{u}}\in \partial\varphi(\widehat{u}).
\end{equation}
 Since $\widehat{u}\in \mathcal{N}$, we have 
 \begin{equation}\label{4eq10}
0=\langle J^{'}(\widehat{u}),\widehat{u}\rangle+\lambda_0\langle \varphi^{*}_{\widehat{u}},\widehat{u}\rangle =\lambda_0\langle \varphi^{*}_{\widehat{u}},\widehat{u}\rangle,\ \ \text{for all}\ \varphi^*_{\widehat{u}}\in \partial\varphi(\widehat{u})
\end{equation}
Using \eqref{4eq500}, hypotheses $(f_4)$, $(g_3)$ and the fact that $\widehat{u}\in\mathcal{N}$, we get
\begin{align}\label{4eq12}
 \langle \varphi^{*}_{\widehat{u}},\widehat{u}\rangle
 & = \int_{\mathbb{R}^{d}}\int_{\mathbb{R}^{d}}g^{'}\left( \frac{\widehat{u}(x)-\widehat{u}(y)}{\vert x-y\vert^{\alpha}}\right)\frac{[\widehat{u}(x)-\widehat{u}(y)]^{2}}{\vert x-y\vert^{2\alpha}}\frac{dxdy}{\vert x-y\vert^{d}}\nonumber\\
& +  \int_{\mathbb{R}^{d}}g^{'}(\widehat{u})\widehat{u}^{2} dx - \int_{\mathbb{R}^{d}} K(x)f^{*}(x,\widehat{u})\widehat{u}^{2} dx\nonumber\\
& \leq [g^+-1] \left[ \int_{\mathbb{R}^{d}}\int_{\mathbb{R}^{d}}g\left( \frac{\widehat{u}(x)-\widehat{u}(y)}{\vert x-y\vert^{\alpha}}\right)\frac{\widehat{u}(x)-\widehat{u}(y)}{\vert x-y\vert^{\alpha}}\frac{dxdy}{\vert x-y\vert^{d}}\right.\nonumber\\
&\left.+   \int_{\mathbb{R}^{d}}g(\widehat{u})\widehat{u}dx\right]  - \int_{\mathbb{R}^{d}} K(x)f^{*}(x,\widehat{u})\widehat{u}^{2} dx\nonumber\\
& = \int_{\mathbb{R}^{d}} K(x)\left([g^+-1]f(x,\widehat{u})\widehat{u}- f^{*}(x,\widehat{u})\widehat{u}^{2} \right) dx\nonumber\\
& <0
\end{align}
According to \eqref{4eq10}, $\lambda_0=0$. Therefore, from \eqref{4eq11}, we deduce that 
$$J^{'}(\widehat{u})=0\ \text{in}\ (W^{\alpha,G}(\mathbb{R}^d))^*.$$
Hence, $\widehat{u}$ is a critical point of $J$, so, it is a weak solution of problem \eqref{P}.\\
Thus the proof.
\end{proof}
\section{Least energy nodal solution}
In this section, we establish the existence of least energy nodal solution for problem \eqref{P} and we show that the ground state solution $\widehat{u}$, obtained in Proposition \ref{4prop9}, is of fixed sign. Finally, we give the proof of Theorem \ref{4thm1}.\\

In order to find a least energy nodal solution for problem \eqref{P}, we look for a minimizer of the energy functional $J$ on the constraint  $\mathcal{M}$. Let consider the following minimization problem 
\begin{equation}\label{M1}
m_1=\inf_{\mathcal{M}}J.\tag{M-1}
\end{equation}

\begin{prop}\label{4prop11}
Assume that the hypotheses  $(H_f)$, $(H_G)$ and $(H_K)$ hold. Let $w\in W^{\alpha,G}(\mathbb{R}^{d})$ such that $w^{\pm}\neq 0,$ then there exists a unique pair $t_{w^+},s_{w^-}>0$ such that 
$$t_{w^+}w^++s_{w^-}w^-\in\mathcal{M}.$$
\end{prop}
\begin{proof}
Let $\xi:\ (0,+\infty)\times  (0,+\infty) \rightarrow \mathbb{R}^{2}$ be a continuous vector field given by
$$\xi(t,s)=\big{(}\xi_{1}(t,s),\xi_{2}(t,s)\big{)},\ \ \text{for all}\ \ t,s \in (0,+\infty)\times (0,+\infty)$$
where $$\xi_{1}(t,s)=\langle J'(tw^{+}+sw^{-}),tw^{+}\rangle\ \text{ and}\ \xi_{2}(t,s)=\langle J'(tw^{+}+sw^{-}),sw^{-}\rangle.$$
Arguing as in the proof of Proposition \ref{4prop2}, there exist $r_1>0$ small enough and $R_1>0$ large enough such that
$$\xi_1(t,t)>0,\ \ \xi_2(t,t)>0,\ \ \text{for all}\ t\in(0,r_1),$$
\begin{equation}\label{4eq90}
\xi_1(t,t)<0,\ \ \xi_2(t,t)<0,\ \ \text{for all}\ t\in(R_1,+\infty).
\end{equation}
Note that $\xi_1(t,s)$ is non-decreasing in $s$ on $(0,+\infty)$ for fixed $t>0$ and $\xi_2(t,s)$ is non-decreasing in $t$ on $(0,+\infty)$ for fixed $s>0$ (see \cite[Proof of Lemma 4.7]{47}). Then, there are $r>0$, $R>0$ with $r<R$ such that 
$$\xi_1(r,s)>0,\ \ \xi_1(R,s)<0,\ \ \text{for all}\ s\in(r,R],$$
$$\xi_2(t,r)>0,\ \ \xi_2(t,R)<0,\ \ \text{for all}\ t\in(r,R].$$
Applying the Miranda's theorem \cite{412} on $\xi$, there exist some $t_{w^+},s_{w^-} \in (r,R]$ such that $\xi_1(t_{w^+},s_{w^-})=\xi_2(t_{w^+},s_{w^-})=0$. Which means that $t_{w^+}w^++s_{w^-}w^-\in\mathcal{M}.$

For the uniqueness of the pairs $(t_{w^+},s_{w^-})$, we argue by contradiction. Suppose that there exist two different pairs $(t_1,s_1)$ and $(t_2,s_2)$ such that 
$$t_1w^++s_1w^-\in \mathcal{M}\ \text{and}\ t_2w^++s_2w^-\in \mathcal{M}.$$
We distinguish two cases :\\
\textbf{(A)}: If $w\in \mathcal{M}$. Without loss of generality, we may take $(t_1,s_1)=(1,1)$ and assume that $t_2\leq s_2$, we have 
\begin{align}\label{4eq210}
\int_{\mathbb{R}^d}K(x)f(x,w^+)w^+ dx& =A^+(w)
\end{align}
and 
\begin{align}\label{4eq211}
\int_{\mathbb{R}^d}K(x)f(x,w^-)w^- dx& =A^-(w).
\end{align}
Where 
\begin{align}\label{4eq220}
A^+(w)&=\int_{\text{supp}(w^+)}\int_{\text{supp}(w^+)}g\left( \frac{w^+(x)-w^+(y)}{\vert x-y\vert^{\alpha}}\right)\frac{w^+(x)-w^+(y)}{\vert x-y\vert^{\alpha +d}}dxdy\nonumber\\
& + \int_{\text{supp}(w^-)}\int_{\text{supp}(w^+)}g\left( \frac{w^+(x)-w^-(y)}{\vert x-y\vert^{\alpha}}\right)\frac{w^+(x)}{\vert x-y\vert^{\alpha +d}}dxdy\nonumber\\
& +\int_{\text{supp}(w^+)}\int_{\text{supp}(w^-)}g\left( \frac{w^-(x)-w^+(y)}{\vert x-y\vert^{\alpha}}\right)\frac{-w^+(y)}{\vert x-y\vert^{\alpha +d}}dxdy\nonumber\\
& +\int_{\mathbb{R}^d}g(w^+)w^+ dx
\end{align}
and 
\begin{align}\label{4eq2221}
A^-(w)&=\int_{\text{supp}(w^-)}\int_{\text{supp}(w^-)}g\left( \frac{w^-(x)-w^-(y)}{\vert x-y\vert^{\alpha}}\right)\frac{w^-(x)-w^-(y)}{\vert x-y\vert^{\alpha +d}}dxdy\nonumber\\
& + \int_{\text{supp}(w^-)}\int_{\text{supp}(w^+)}g\left( \frac{w^+(x)-w^-(y)}{\vert x-y\vert^{\alpha}}\right)\frac{-w^-(y)}{\vert x-y\vert^{\alpha +d}}dxdy\nonumber\\
& +\int_{\text{supp}(w^+)}\int_{\text{supp}(w^-)}g\left( \frac{w^-(x)-w^+(y)}{\vert x-y\vert^{\alpha}}\right)\frac{w^-(x)}{\vert x-y\vert^{\alpha +d}}dxdy\nonumber\\
& +\int_{\mathbb{R}^d}g(w^-)w^- dx.
\end{align}
Since the map $s\mapsto g(s)$ is non-decreasing on $(0,+\infty)$ and on $(-\infty,0)$ and $t_2\leq s_2$, we infer that
\begin{equation}
\left\lbrace  
\begin{array}{l}
g\left( \frac{t_2w^+(x)-s_2w^-(y)}{\vert x-y\vert^{\alpha}}\right)t_2w^+(x) \geq g\left( \frac{t_2w^+(x)-t_2w^-(y)}{\vert x-y\vert^{\alpha}}\right)t_2w^+(x),\\
\ \\
 g\left( \frac{s_2w^-(x)-t_2w^+(y)}{\vert x-y\vert^{\alpha}}\right)\left( -t_2w^+(y)\right) \geq g\left( \frac{t_2w^-(x)-t_2w^+(y)}{\vert x-y\vert^{\alpha}}\right)\left( -t_2w^+(y)\right), \\
 \ \\
g\left( \frac{t_2w^+(x)-s_2w^-(y)}{\vert x-y\vert^{\alpha}}\right)\left( -s_2w^-(y)\right) \leq g\left( \frac{s_2w^+(x)-s_2w^-(y)}{\vert x-y\vert^{\alpha}}\right)\left( -s_2w^-(y)\right), \\
\ \\
g\left( \frac{s_2w^-(x)-t_2w^+(y)}{\vert x-y\vert^{\alpha}}\right)s_2w^-(x)\leq g\left( \frac{s_2w^-(x)-s_2w^+(y)}{\vert x-y\vert^{\alpha}}\right)s_2w^-(x),\\
\end{array}
\right.
\end{equation}
for a.a. $x,y\in \mathbb{R}^d$.\\
By Lemma \ref{4lem1} and since $t_1w^++s_1w^-\in \mathcal{M},\ \text{and}\ t_2w^++s_2w^-\in \mathcal{M}$, it yields that
\begin{align}\label{4eq212}
\int_{\mathbb{R}^d}K(x)\frac{f(x,t_2w^+)t_2w^+}{\min\lbrace t_2^{g^-},t_2^{g^+}\rbrace} dx& \geq  A^+(w)
\end{align}
and 
\begin{align}\label{4eq213}
\int_{\mathbb{R}^d}K(x)\frac{f(x,s_2w^-)s_2w^-}{\max\lbrace s_2^{g^-},s_2^{g^+}\rbrace} dx& \leq A^-(w).
\end{align}
In what follows, we will show that the following five cases cannot happen:
\begin{enumerate}
\item[$(1)$] $t_2 < s_2=1.$
\item[$(2)$] $s_2> t_2=1.$
\item[$(3)$] $0<t_2\leq s_2<1.$
\item[$(4)$] $1<t_2\leq s_2.$
\item[$(5)$] $0<t_2<1< s_2.$
\end{enumerate} 
 Suppose that one of the cases $(1)$, $(3)$ or $(5)$, holds. According to \eqref{4eq50}, \eqref{4eq210}, \eqref{4eq212} and hypothesis $(K_1)$, we get  
\begin{align*}
0\leq \int_{\mathbb{R}^d}K(x)\vert w^+\vert^{g^+-1}w^+\left[\frac{f(x,t_2w^+)}{\vert t_2w^+\vert^{g^+-1}}-\frac{f(x,w^+)}{\vert w^+\vert^{g^+-1}} \right]dx <0.
\end{align*}
Thus the contradiction. Then, the cases $(1)$, $(3)$ and $(5)$ cannot be realized.\\
Suppose that case $(2)$ or $(4)$ holds. According to \eqref{4eq50}, \eqref{4eq211}, \eqref{4eq213} and hypothesis $(K_1)$, one has
\begin{align*}
0\leq \int_{\mathbb{R}^d}K(x)\vert w^-\vert^{g^+-1}w^-\left[\frac{f(x,w^-)}{\vert w^-\vert^{g^+-1}} -\frac{f(x,s_2w^-)}{\vert s_2w^-\vert^{g^+-1}}\right]dx <0.
\end{align*}
Which is a contradiction too. Then, the cases $(2)$ and $(4)$ cannot be realized.
We deduce that $(t_1,s_1)=(1,1)=(t_2,s_2)$. \\
\textbf{(B)}: If $w \notin \mathcal{M}$. Let $v=t_1w^++s_1w^-\in \mathcal{M}$, $v^+=t_1w^+$ and $v^-=s_1w^-$, so $(t_1,s_1)\neq (1,1)$. It is clear that
$$t_2w^++s_2w^-=\frac{t_2}{t_1}t_1w^++\frac{s_2}{s_1}s_1w^-=\frac{t_2}{t_1}v^++\frac{s_2}{s_1}v^-\in \mathcal{M}.$$ 
Arguing as in the case \textbf{(A)}, we conclude that 
$$\frac{t_2}{t_1}=\frac{s_2}{s_1}=1.$$ 
This completes the proof.
\end{proof}

\begin{rem}\label{4prop111}
Under the hypotheses $(H_f)$, $(H_G)$ and $(H_K)$, 
$$0<m_0=\inf_{\mathcal{N}}J\leq \inf_{\mathcal{M}}J=m_{1}.$$
\end{rem}
\begin{prop}\label{4prop16}
Assume that the hypotheses $(H_f)$, $(H_G)$ and $(H_K)$ hold. Then, for all $w\in\mathcal{M}$, 
$$J(tw^++sw^-)\leq J(w),\  \text{for all}\ t,s> 0.$$
\end{prop}
\begin{proof}
Let $w\in\mathcal{M}$ and consider the fibering map $\mu_w:(0,+\infty)\times (0,+\infty)\longrightarrow \mathbb{R}$ defined by 
$$\mu_w(t,s)=J(tw^++sw^-)\ \ \text{for all}\ t,s> 0.$$
In light of Proposition \ref{4prop11},
\begin{equation}\label{4eq99}
\mu_w(0,0)=J(0)=0<m_1\leq \mu_w(1,1)=J(w).
\end{equation}
Let $t,s>0$ large enough, using Lemmas \ref{4lem1} and \ref{4lem2}, we obtain
\begin{align}\label{4eq18}
\mu_w(t,s)&\leq \max\left\lbrace \Vert tw^++sw^-\Vert^{g^-},\Vert tw^++sw^-\Vert^{g^+}\right\rbrace -\int_{\mathbb{R}^d} K(x)F(x,tw^++sw^-)dx\nonumber\\
& \leq 2^{g^+-1}\max\left\lbrace \vert t \vert^{g^-}\Vert w^+\Vert^{g^-}+\vert s \vert^{g^-}\Vert w^-\Vert^{g^-},\vert t \vert^{g^+}\Vert w^+\Vert^{g^+}+\vert s \vert^{g^+}\Vert w^-\Vert^{g^+}\right\rbrace\nonumber\\
& -\int_{\mathbb{R}^d} K(x)F(x,tw^++sw^-)dx\nonumber\\
& \leq 2^{g^+-1}\max\left\lbrace\max\lbrace\vert t \vert^{g^-},\vert s \vert^{g^-}\rbrace\left( \Vert w^+\Vert^{g^-}+\Vert w^-\Vert^{g^-}\right),\max\lbrace\vert t \vert^{g^+},\vert s \vert^{g^+}\rbrace\left( \Vert w^+\Vert^{g^+}+\Vert w^-\Vert^{g^+}\right) \right\rbrace\nonumber\\
& -\int_{\mathbb{R}^d} K(x)F(x,tw^++sw^-)dx\nonumber\\
&\leq 2^{g^+-1}\max\lbrace\vert t \vert^{g^+},\vert s \vert^{g^+}\rbrace\max\left\lbrace \Vert w^+\Vert^{g^-}+\Vert w^-\Vert^{g^-}, \Vert w^+\Vert^{g^+}+\Vert w^-\Vert^{g^+} \right\rbrace\nonumber\\
& -\int_{\mathbb{R}^d} K(x)F(x,tw^++sw^-)dx.
\end{align}
It follows that 
\begin{align}\label{4eq19}
\frac{\mu_w(t,s)}{\max\lbrace\vert t \vert^{g^+},\vert s \vert^{g^+}\rbrace}&\leq 2^{g^+-1}\max\left\lbrace \Vert w^+\Vert^{g^-}+\Vert w^-\Vert^{g^-}, \Vert w^+\Vert^{g^+}+\Vert w^-\Vert^{g^+} \right\rbrace\nonumber\\ &-\int_{\mathbb{R}^d} K(x)\frac{F(x,tw^++sw^-)}{\max\lbrace\vert t \vert^{g^+},\vert s \vert^{g^+}\rbrace}dx.
\end{align}
By assumption $(f_3)$ and the fact that $\text{supp}(w^+)\cap\text{supp}(w^-)=\emptyset$, we infer that 
\begin{equation}\label{4eq20}
\lim\limits_{\vert (t,s)\vert\rightarrow +\infty}\frac{F(x,tw^++sw^-)}{\max\lbrace\vert t \vert^{g^+},\vert s \vert^{g^+}\rbrace}=+\infty,\ \ \text{for a.a.}\ x\in\mathbb{R}^d.
\end{equation}
Applying \eqref{4eq19} and \eqref{4eq20}, we deduce that
 $$\limsup\limits_{\vert (t,s)\vert\rightarrow +\infty} \mu_w(t,s)\leq -\infty.$$
According to \eqref{4eq99}, the map $\mu_w(\cdot,\cdot)$ has  a global maximum $(t_{w^+},s_{w^-})\in (0,+\infty)\times (0,+\infty)$.  $(t_{w^+},s_{w^-})$ is a critical point for $\mu_w(\cdot,\cdot)$, that is, 
 
$$\langle J^{'}(t_{w^+}w^++s_{w^-}w^-),w^+\rangle=0$$
and 
$$\langle J^{'}(t_{w^+}w^++s_{w^-}w^-),w^-\rangle=0.$$
By  Proposition \ref{4prop16} and the fact that $w\in \mathcal{M}$,
$$(t_{w^+},s_{w^-})=(1,1).$$
Hence,
$$J(tw^++sw^-)\leq J(t_{w^+}w^++s_{w^-}w^-)= J(w),\ \text{for all}\ t,s>0.$$
This ends the proof.
\end{proof}
\begin{prop}\label{4prop12} Assume that hypotheses $(H_f)$, $(H_G)$ and $(H_K)$ hold.
Let $\lbrace w_{n}\rbrace _{n} \subset \mathcal{M}$ such that $w_{n}\rightharpoonup w $ in $W^{\alpha,G}(\mathbb{R}^{d})$, then $w^{\pm}\neq 0.$
\end{prop}
\begin{proof}[Proof] We claim that there is $\varrho>0$ such that 
\begin{equation}\label{4eq7845}
\varrho\leq \Vert v^\pm\Vert,\ \ \text{for all}\ v\in \mathcal{M}.
\end{equation}
Indeed, by using Lemma \ref{4lem5},
\begin{align*}
& \int_{\mathbb{R}^{d}}\int_{\mathbb{R}^{d}}g\left( \frac{v^{\pm}(x)-v^{\pm}(y)}{\vert x-y\vert^{\alpha}}\right)\frac{v^{\pm}(x)-v^{\pm}(y)}{\vert x-y\vert^{\alpha +d}}dxdy+\int_{\mathbb{R}^{d}}g(v^{\pm})v^{\pm}dx \leq\int_{\mathbb{R}^{d}}K(x)f(x,v^{\pm})v^{\pm}dx.
\end{align*}
Exploiting \eqref{2.7}, \eqref{4eq81}, hypotheses $(g_2)$ and $(K_1)$, for all $\varepsilon>0$ we get
\begin{align*}
&[ g^--g^+\varepsilon\Vert K\Vert_{\infty} ]\left[ \int_{\mathbb{R}^{d}}\int_{\mathbb{R}^{d}}G\left( \frac{v^\pm(x)-v^\pm(y)}{\vert x-y\vert^{\alpha}}\right)\frac{dxdy}{\vert x-y\vert^{d}}dxdy+ \int_{\mathbb{R}^{d}}G(v^\pm)dx \right]\\
& \leq g^+_{*}C_{\varepsilon}\Vert K\Vert_{\infty}\int_{\mathbb{R}^{d}}G_{*}(v^\pm)dx.
\end{align*}
Without lose of generality, we may assume that
 $0\neq \Vert v\Vert <1$. By Lemma \ref{4lem2} and Theorem \ref{4thm2}, we deduce that
\begin{equation}\label{600}
[ g^--g^+\varepsilon\Vert K\Vert_{\infty} ]\Vert v^\pm\Vert^{g^+}\leq g^+_{*}\tilde{C}C_{\varepsilon}\Vert K\Vert_{\infty}\Vert v^\pm\Vert^{g^+_{*}}.
\end{equation}
Choosing $\varepsilon$ small enough, we conclude that
\begin{align*}
\left( \frac{C_{1}}{C_{2}}\right)^{\frac{1}{g^+_{*}-g^+}}& \leq \Vert v^\pm\Vert,
\end{align*}
where $C_{1}=g^--g^+\varepsilon\Vert K\Vert_{\infty}>0$ and
$C_{2}=g^+_{*}\tilde{C}C_{\varepsilon}\Vert K\Vert_{\infty}>0.$
  Consequently, there exists a positive radius $\varrho>0$ such that $\Vert v^\pm\Vert\geq \varrho,$ with $\varrho=\left( \frac{C_{1}}{C_{2}}\right)^{\frac{1}{g^+_{*}-g^+}}.$ Thus, the claim.
  
So, by \eqref{4eq7845}, 
\begin{equation}\label{3.7}
\Vert w_{n}^{\pm}\Vert \geq\varrho,\ \ \text{for all}\ n\in\mathbb{N}.
\end{equation}
According to Lemma \ref{4lem5}, 
$$\langle J^{'}(w_{n}^{\pm}),w_{n}^{\pm} \rangle \leq\langle J^{'}(w_{n}),w_{n}^{\pm} \rangle =0.$$
 By $(g_2)$ and Lemma \ref{4lem2}, we get
\begin{equation}\label{3.6}
 g^-\min\lbrace \Vert w_{n}^{\pm}\Vert^{g^-},\Vert w_{n}^{\pm}\Vert^{g^+}\rbrace\leq \int_{\mathbb{R}^{d}}K(x)f(x,w_{n}^{\pm})w_{n}^{\pm}\ dx.
 \end{equation}
 Putting together $(\ref{3.7})$ and $(\ref{3.6}),$ we find
 \begin{equation}\label{3.8}
  g^-\min\lbrace  \varrho^{g^-}, \varrho^{g^+}\rbrace\leq g^-\min\lbrace \Vert w_{n}^{\pm}\Vert^{g^-},\Vert w_{n}^{\pm}\Vert^{g^+}\rbrace\leq \int_{\mathbb{R}^{d}}K(x)f(x,w_{n}^{\pm})w_{n}^{\pm}\ dx.
 \end{equation}
 On the other hand, in light of Lemma \ref{4lem4}, one has
 \begin{equation}\label{3.9}
 \lim_{n\rightarrow+\infty}\int_{\mathbb{R}^{d}}K(x)f(x,w_{n}^{\pm})w_{n}^{\pm}\ dx= \int_{\mathbb{R}^{d}}K(x)f(x,w^{\pm})w^{\pm}\ dx.
 \end{equation}
Combining $(\ref{3.8})$ with $(\ref{3.9}),$ we get
 $$0< g^-\min\lbrace  \varrho^{g^-}, \varrho^{g^+}\rbrace\leq \int_{\mathbb{R}^{d}}K(x)f(x,w^{\pm})w^{\pm}\ dx,$$
 thus, $w^{\pm}\neq 0$. This ends the proof.
\end{proof}

In the following proposition, we prove that the infimum of $J$ is attained on $\mathcal{M}$.
\begin{prop}\label{4prop15}
Assume that the hypotheses $(H_f)$, $(H_G)$ and $(H_K)$ hold. Then, there exists $\widehat{w}\in\mathcal{M}$ such that $J(\widehat{w})=m_1$.
\end{prop}
\begin{proof}
Let $\lbrace w_n\rbrace_{n\in\mathbb{N}}\subset \mathcal{M}$ such that 
$$J(w_n)\mathop{\longrightarrow}_{n\rightarrow +\infty} m_1.$$
Arguing as in the proof of Proposition \ref{4prop5}, we deduce that $\lbrace w_n\rbrace_{n\in\mathbb{N}}$ is  bounded in $W^{\alpha,G}(\mathbb{R}^d)$.\\
By passing to a subsequence if necessary,
$$w_n\rightharpoonup \widehat{w}\ \text{in}\ W^{\alpha,G}(\mathbb{R}^d),$$
\begin{equation}\label{4eq22}
w_n(x)\rightarrow \widehat{w}(x)\ \text{as}\ n\rightarrow +\infty,\  \text{for a.a.}\ x\in\mathbb{R}^d
\end{equation}
and
\begin{equation}\label{4eq221}
w_n^\pm(x)\rightarrow \widehat{w}^\pm(x)\ \text{as}\ n\rightarrow +\infty,\  \text{for a.a.}\ x\in\mathbb{R}^d.
\end{equation}
Applying Proposition \ref{4prop12}, we see that
$\widehat{w}^\pm \neq 0.$ So, according to Proposition \ref{4prop11}, there is a unique pair $t_{\widehat{w}^+},s_{\widehat{w}^-}>0$ such that $t_{\widehat{w}^+}\widehat{w}^{+}+s_{\widehat{w}^-}\widehat{w}^{-}\in \mathcal{M}$, that is,
\begin{equation}\label{4eq13}
\langle J^{'}(t_{\widehat{w}^+}\widehat{w}^{+}+s_{\widehat{w}^-}\widehat{w}^{-}),\widehat{w}^+\rangle=0\ \text{and}\ \langle J^{'}(t_{\widehat{w}^+}\widehat{w}^{+}+s_{\widehat{w}^-}\widehat{w}^{-}),\widehat{w}^-\rangle=0.
\end{equation}
Since $\lbrace w_n\rbrace_{n\in\mathbb{N}}\subset \mathcal{M}$, by \eqref{4eq221}, Proposition \ref{4prop16}, Lemma \ref{4lem4} and Fatou's lemma, we obtain
\begin{align}\label{4eq23}
m_1=\lim_{n\rightarrow +\infty}J(w_n)& \geq \liminf_{n\rightarrow +\infty}J(t_{\widehat{w}^+}w^+_n+s_{\widehat{w}^-}w^-_n)\nonumber\\
& \geq J(t_{\widehat{w}^+}\widehat{w}^++s_{\widehat{w}^-}\widehat{w}^-)\nonumber\\
& \geq m_1.
\end{align}
Therefore,
\begin{equation}\label{4eq24}
 m_1=\inf\limits_{\mathcal{M}}J=J(t_{\widehat{w}^+}\widehat{w}^++s_{\widehat{w}^-}\widehat{w}^-).
\end{equation}

 We show that $t_{\widehat{w}^+}=s_{\widehat{w}^-}=1$ and we produced it in two steps.\\
 
 Step 1: $0<t_{\widehat{w}^+},s_{\widehat{w}^-}\leq 1$. Indeed, using \eqref{4eq221}, Lemma \ref{4lem4} and Fatou's lemma, we find that
 \begin{align}\label{4eq150}
 \int_{\mathbb{R}^d}\int_{\mathbb{R}^d}g\left( \frac{\widehat{w}(x)-\widehat{w}(y)}{\vert x-y\vert^\alpha}\right) \frac{\widehat{w}^\pm(x)-\widehat{w}^\pm(y)}{\vert x-y\vert^{\alpha+d}}dxdy+\int_{\mathbb{R}^d}g(\widehat{w}^\pm)\widehat{w}^\pm dx\leq \int_{\mathbb{R}^d}K(x)f(x,\widehat{w}^\pm)\widehat{w}^\pm dx.
 \end{align}
From \eqref{4eq13}, we have
 \begin{align}\label{4eq200}
 & \int_{\mathbb{R}^{d}}K(x)f(x,t_{\widehat{w}^+}\widehat{w}^{+})t_{\widehat{w}^+}\widehat{w}^{+}\ dx \nonumber\\ 
 & =\int_{\text{supp}(\widehat{w}^+)}\int_{\text{supp}(\widehat{w}^+)}g\left( \frac{t_{\widehat{w}^+}\widehat{w}^{+}(x)-t_{\widehat{w}^+}\widehat{w}^{+}(y)}{\vert x-y\vert^{\alpha}}\right) \frac{t_{\widehat{w}^+}\widehat{w}^{+}(x)-t_{\widehat{w}^+}\widehat{w}^{+}(y)}{\vert x-y\vert^{\alpha +d}}dxdy\nonumber\\
 & + \int_{\text{supp}(\widehat{w}^-)}\int_{\text{supp}(\widehat{w}^+)}g\left( \frac{t_{\widehat{w}^+}\widehat{w}^{+}(x)-s_{\widehat{w}^-}\widehat{w}^{-}(y)}{\vert x-y\vert^{\alpha}}\right) \frac{t_{\widehat{w}^+}\widehat{w}^{+}(x)}{\vert x-y\vert^{\alpha +d}}dxdy\nonumber\\
 & +\int_{\text{supp}(\widehat{w}^+)}\int_{\text{supp}(\widehat{w}^-)}g\left( \frac{s_{\widehat{w}^-}\widehat{w}^{-}(x)-t_{\widehat{w}^+}\widehat{w}^{+}(y)}{\vert x-y\vert^{\alpha}}\right) \frac{-t_{\widehat{w}^+}\widehat{w}^{+}(y)}{\vert x-y\vert^{\alpha +d}}dxdy\nonumber\\
 & + \int_{\mathbb{R}^{d}}g(t_{\widehat{w}^+}\widehat{w}^{+})t_{\widehat{w}^+}\widehat{w}^{+}\ dx.
 \end{align}
Without loss of generality, we suppose that $t_{\widehat{w}^+}\geq s_{\widehat{w}^-}$. By \eqref{4eq200}, Lemma \ref{4lem1} and the fact that the map $s\mapsto g(s)$ is non-decreasing function on $\mathbb{R}$, it yields that
 \begin{align*}
 &\int_{\mathbb{R}^{d}}K(x)f(x,t_{\widehat{w}^+}\widehat{w}^{+})t_{\widehat{w}^+}\widehat{w}^{+}\ dx
  \leq \max\lbrace t_{\widehat{w}^+}^{g^-},t_{\widehat{w}^+}^{g^+}\rbrace \sigma,
 \end{align*}
 where
 \begin{align*}
 0\leq\sigma & =\int_{\text{supp}(\widehat{w}^+)}\int_{\text{supp}(\widehat{w}^+)}g\left( \frac{\widehat{w}^{+}(x)-\widehat{w}^{+}(y)}{\vert x-y\vert^{\alpha}}\right) \frac{\widehat{w}^{+}(x)-\widehat{w}^{+}(y)}{\vert x-y\vert^{\alpha +d}}dxdy\\
 & +\int_{\text{supp}(\widehat{w}^-)}\int_{\text{supp}(\widehat{w}^+)}g\left( \frac{\widehat{w}^{+}(x)-\widehat{w}^{-}(y)}{\vert x-y\vert^{\alpha}}\right) \frac{\widehat{w}^{+}(x)}{\vert x-y\vert^{\alpha +d}}dxdy\\
 & +\int_{\text{supp}(\widehat{w}^+)}\int_{\text{supp}(\widehat{w}^-)}g\left( \frac{\widehat{w}^{-}(x)-\widehat{w}^{+}(y)}{\vert x-y\vert^{\alpha}}\right) \frac{-\widehat{w}^{+}(y)}{\vert x-y\vert^{\alpha +d}}dxdy\\
 & + \int_{\mathbb{R}^{d}}g(\widehat{w}^{+})\widehat{w}^{+}\ dx\\
 &= \int_{\mathbb{R}^{d}}\int_{\mathbb{R}^{d}}g\left( \frac{\widehat{w}(x)-\widehat{w}(y)}{\vert x-y\vert^{\alpha}}\right) \frac{\widehat{w}^{+}(x)-\widehat{w}^{+}(y)}{\vert x-y\vert^{\alpha +d}}dxdy\\
 & + \int_{\mathbb{R}^{d}}g(\widehat{w}^{+})\widehat{w}^{+}\ dx.
 \end{align*}
Arguing by contradiction, and suppose that $t_{\widehat{w}^+}>1,$ then
 \begin{equation}\label{100}
 \int_{\mathbb{R}^{d}}K(x)\frac{f(x,t_{\widehat{w}^+}\widehat{w}^{+})t_{\widehat{w}^+}\widehat{w}^{+}}{t_{\widehat{w}^+}^{g^+}}\leq  \sigma.
 \end{equation}
 Putting together \eqref{4eq150} and \eqref{100}, using $(K_1)$ and \eqref{4eq50}, we obtain
 \begin{align*}
0  & \leq \int_{\mathbb{R}^{d}}K(x)(\widehat{w}^{+})^{g^+}\left[ \frac{f(x,\widehat{w}^{+})}{ (\widehat{w}^{+})^{g^+-1}}-\frac{f(x,t_{\widehat{w}^+}\widehat{w}^{+})}{ (t_{\widehat{w}^+}\widehat{w}^{+})^{g^+-1}}\right] dx < 0
 \end{align*}
which is  a contradiction. Therefore, $0<t_{\widehat{w}^+},s_{\widehat{w}^-} \leq 1.$\\

Step 2: $t_{\widehat{w}^+}=s_{\widehat{w}^-}=1$. Indeed, we argue by contradiction and suppose that $(t_{\widehat{w}^+},s_{\widehat{w}^-})\neq (1,1).$\\
By \eqref{4eq51}, \eqref{4eq22}, \eqref{4eq221}, Lemma \ref{4lem3}  and Fatou's lemma, it follows that
\begin{align*}
m_1 &\leq J(t_{\widehat{w}^+}\widehat{w}^++s_{\widehat{w}^-}\widehat{w}^-)\ \ (\text{since}\ t_{\widehat{w}^+}\widehat{w}^++s_{\widehat{w}^-}\widehat{w}^-\in \mathcal{M})\nonumber\\ &=J(t_{\widehat{w}^+}\widehat{w}^++s_{\widehat{w}^-}\widehat{w}^-)-\frac{1}{g^+}\langle J^{'}(t_{\widehat{w}^+}\widehat{w}^++s_{\widehat{w}^-}\widehat{w}^-),t_{\widehat{w}^+}\widehat{w}^++s_{\widehat{w}^-}\widehat{w}^-\rangle\nonumber\\
& < J(\widehat{w})-\frac{1}{g^+}\langle J^{'}(\widehat{w}),\widehat{w}\rangle\nonumber\\
& \leq \liminf_{n\rightarrow +\infty}\left[  J(w_n)-\frac{1}{g^+}\langle J^{'}(w_n),w_n\rangle\right] \nonumber\\
&=\liminf_{n\rightarrow +\infty} J(w_n)\ \ (\text{since}\ w_n\in \mathcal{M})\nonumber\\
& = m_1.
\end{align*}
Which is a contradiction, thus, $t_{\widehat{w}^+}=s_{\widehat{w}^-}=1$. Hence, according to \eqref{4eq24}, it comes that
$$m_1=\inf\limits_{\mathcal{M}}J=J(t_{\widehat{w}^+}\widehat{w}^++s_{\widehat{w}^-}\widehat{w}^-)=J(\widehat{w}).$$
This ends the proof.
\end{proof}
\begin{prop}\label{4prop14}
Under the hypotheses $(H_f)$, $(H_G)$ and $(H_K)$, $\widehat{w}$ is a critical point for $J$, that is, $\widehat{w}$ is a least energy weak nodal solution of problem \eqref{P}.
\end{prop}
\begin{proof}[Proof]
We consider the functionals $\varphi_\pm:W^{\alpha,G}(\mathbb{R}^d)\rightarrow \mathbb{R}$ defined by 
\begin{align*}
\varphi_\pm(w)& =\langle J^{'}(w),w^\pm\rangle =\int_{\mathbb{R}^{d}}\int_{\mathbb{R}^{d}}g\left( \frac{w(x)-w(y)}{\vert x-y\vert^{\alpha}}\right)\frac{w^\pm(x)-w^\pm(y)}{\vert x-y\vert^{\alpha +d}}dxdy\\
& + \int_{\mathbb{R}^{d}}g(w)w^\pm dx - \int_{\mathbb{R}^{d}} K(x)f(x,w^\pm)w^\pm dx.
\end{align*}
 $\varphi_\pm$  is locally Lipschitz (see\cite[Theorem 2.7.2, p. 221]{419}).\\
Let $w\in W^{\alpha,G}(\mathbb{R}^d)$. For all $\varphi^*_{w^\pm}\in \partial \varphi_\pm(w)$, there is $f^*(x,w^\pm)\in \partial_{w^\pm} f(x,w^\pm)$ such that 
\begin{align}\label{4eq400}
\langle\varphi_{w^\pm}^{*},v\rangle &  =\int_{\mathbb{R}^{d}}\int_{\mathbb{R}^{d}}g\left( \frac{w(x)-w(y)}{\vert x-y\vert^{\alpha}}\right)\frac{v^\pm(x)-v^\pm(y)}{\vert x-y\vert^{\alpha+d}}dxdy\nonumber\\
& + \int_{\mathbb{R}^{d}}g(w^\pm)v^\pm dx - \int_{\mathbb{R}^{d}} K(x)f(x,w^\pm)v^\pm dx\nonumber\\
& + \int_{\mathbb{R}^{d}}\int_{\mathbb{R}^{d}}g^{'}\left( \frac{w(x)-w(y)}{\vert x-y\vert^{\alpha}}\right)\frac{[v(x)-v(y)][w^\pm(x)-w^\pm(y)]}{\vert x-y\vert^{2\alpha+d}}dxdy\nonumber\\
& + \int_{\mathbb{R}^{d}}g^{'}(w^\pm)(v^\pm)^{2} dx - \int_{\mathbb{R}^{d}} K(x)f^{*}(x,w^\pm)(v^\pm)^{2} dx,
\end{align} 
for all $v\in W^{\alpha,G}(\mathbb{R}^d)$.\\
By Proposition \ref{4prop15}, 
$$J(\widehat{w})=m_1=\inf\left\lbrace J(w):\ w\in W^{\alpha,G}(\mathbb{R}^d),\ w^\pm\neq 0,\ \varphi_+(w)=\varphi_-(w)=0 \right\rbrace.$$
Thus, according to the non-smooth multiplier rule of Clarke \cite[Theorem 10.47, p. 221]{418}, there exist $\lambda_+,\lambda_-\geq 0$ such that
$$
0\in\partial(J+ \lambda_+ \varphi_++\lambda_- \varphi_-)(\widehat{w}).
$$
The subdifferential calculus of Clarke \cite[p. 48]{419}, gives that 

$$0\in\partial J(\widehat{w})+\lambda_+ \partial\varphi_+(\widehat{w})+\lambda_- \partial\varphi_-(\widehat{w}).$$
Then,
\begin{equation}\label{44eq11}
0=J^{'}(\widehat{u})+\lambda_+ \varphi_{w^+}^*+\lambda_- \varphi_{w^-}^*\ \ \text{in}\ (W^{\alpha,G},(\mathbb{R}^d))^*,\ \text{for all}\ \varphi^*_{\widehat{w}^+}\in \partial\varphi_+(\widehat{w})\ \text{and all}\ \varphi^*_{\widehat{w}^-}\in \partial\varphi_-(\widehat{w}).
\end{equation}
 Since $\widehat{w}\in \mathcal{M}$, 
 \begin{equation}\label{44eq10}
0=\langle J^{'}(\widehat{w}),\widehat{w}\rangle+\lambda_+\langle \varphi^{*}_{\widehat{w}^+},\widehat{w}\rangle+\lambda_-\langle \varphi^{*}_{\widehat{w}^-},\widehat{w}\rangle =\lambda_+\langle \varphi^{*}_{\widehat{w}^+},\widehat{w}\rangle+\lambda_-\langle \varphi^{*}_{\widehat{w}^-},\widehat{w}\rangle,
\end{equation}
for all $\varphi^*_{\widehat{w}^+}\in \partial\varphi_+(\widehat{w})$ and all $ \varphi^*_{\widehat{w}^-}\in \partial\varphi_-(\widehat{w}).$\\
Let observe that
\begin{equation}\label{4eq401}
\text{sign}\left( \widehat{w}(x)-\widehat{w}(y)\right) =\text{sign}\left( \widehat{w}^\pm(x)-\widehat{w}^\pm(y)\right) , \ \ \text{for a.a. }\ x,y\in \mathbb{R}^d.
\end{equation}
Using \eqref{4eq400}, \eqref{4eq401}, assumption $(g_3)$ and the fact that $\widehat{w}\in\mathcal{M}$, we obtain
\begin{align}\label{4eq27}
\langle \varphi^{*}_{\widehat{w}^{\pm}},\widehat{w}\rangle
& = \int_{\mathbb{R}^{d}}\int_{\mathbb{R}^{d}}g^{'}\left( \frac{\widehat{w}(x)-\widehat{w}(y)}{\vert x-y\vert^{\alpha}}\right)\frac{[\widehat{w}(x)-\widehat{w}(y)][\widehat{w}^{\pm}(x)-\widehat{w}^{\pm}(y)]}{\vert x-y\vert^{2\alpha+d}}dxdy\nonumber\\
& + \int_{\mathbb{R}^{d}}g^{'}(\widehat{w}^{\pm})(\widehat{w}^{\pm})^{2} dx - \int_{\mathbb{R}^{d}} K(x)f^{'}(x,\widehat{w}^{\pm})(\widehat{w}^{\pm})^{2} dx\nonumber\\
&\leq [g^+-1]\left[ 
 \int_{\mathbb{R}^{d}}\int_{\mathbb{R}^{d}}g\left( \frac{\widehat{w}(x)-\widehat{w}(y)}{\vert x-y\vert^{\alpha}}\right)\frac{\widehat{w}^{\pm}(x)-\widehat{w}^{\pm}(y)}{\vert x-y\vert^{\alpha+d}}dxdy\right.\nonumber\\
& + \left. \int_{\mathbb{R}^{d}}g(\widehat{w}^{\pm})\widehat{w}^{\pm} dx\right] - \int_{\mathbb{R}^{d}} K(x)f^{*}(x,\widehat{w}^{\pm})(\widehat{w}^{\pm})^{2} dx\nonumber\\
& =\int_{\mathbb{R}^{d}} K(x)\left([g^+-1]f(x,\widehat{w}^{\pm})\widehat{w}^{\pm}- f^{*}(x,\widehat{w}^{\pm})(\widehat{w}^{\pm})^{2}\right)dx. 
\end{align}
By \eqref{4eq27}, hypotheses $(f_4)$ and $(K_1)$, we infer that
\begin{equation}\label{4eq28}
\langle \varphi^{*}_{\widehat{w}^+},\widehat{w}\rangle < 0\ \ \text{and}\ \ \langle \varphi^{*}_{\widehat{w}^-},\widehat{w}\rangle <0.
\end{equation}
According to \eqref{44eq10}, it comes that $\lambda_\pm=0.$ Therefore, from \eqref{44eq11}, we conclude that $\widehat{w}$  is a critical point of $J$. Hence, $\widehat{w}$ is a nodal weak solution for problem \eqref{P}.\\
This completes the proof.
\end{proof}
\begin{prop}\label{4prop30}
Assume that the hypotheses $(H_f)$, $(H_G)$ and $(H_K)$ hold. Then, the ground state solution $\widehat{u}$ of problem \eqref{P} has a fixed sign. Moreover, $$m_0=J(\widehat{u})=\inf\limits_{\mathcal{N}}J<\inf\limits_{\mathcal{M}}J=J(\widehat{w})=m_1.$$
\end{prop}
\begin{proof}
We argue by contradiction. Suppose that $\widehat{u}^\pm\neq 0$, then  
\begin{equation}\label{4eq33}
m_0=\inf\limits_{\mathcal{N}}J\geq \inf\limits_{\mathcal{M}}J=m_1.
\end{equation}
Since $\mathcal{M}\subset \mathcal{N}$,
\begin{equation}\label{4eq32}
m_0=\inf\limits_{\mathcal{N}}J\leq \inf\limits_{\mathcal{M}}J=m_1. 
\end{equation}
Putting together \eqref{4eq33} and  \eqref{4eq32}, we get 
\begin{equation}\label{4eq34}
m_0=J(\widehat{u})=\inf\limits_{\mathcal{N}}J= \inf\limits_{\mathcal{M}}J=J(\widehat{w})=m_1.
\end{equation} 
On the other hand, since $\widehat{w}\in\mathcal{M}$, $\widehat{w}^\pm\neq 0$. Then, from Proposition  \ref{4prop2}, there is a unique pair $t_{\widehat{w}^+},s_{\widehat{w}^-}>0$ such that
 $$t_{\widehat{w}^+}\widehat{w}^+\in \mathcal{N}\ \text{and}\ s_{\widehat{w}^-}\widehat{w}^-\in \mathcal{N}.$$
By Lemma \ref{4lem9} and Proposition \ref{4prop16}, it follows that
\begin{align}\label{4eq31}
2m_0 &\leq J(t_{\widehat{w}^+}\widehat{w}^+)+J(s_{\widehat{w}^-}\widehat{w}^-)\nonumber\\
& <J(t_{\widehat{w}^+}\widehat{w}^++s_{\widehat{w}^-}\widehat{w}^-)\nonumber\\
&\leq J(\widehat{w})\nonumber\\
&=\inf\limits_{\mathcal{M}}J=m_1.
\end{align}
Which is a contradiction with \eqref{4eq34}. Therefore, $\widehat{u}$ has a fixed sign, and
$$m_0=J(\widehat{u})=\inf\limits_{\mathcal{N}}J<\inf\limits_{\mathcal{M}}J=J(\widehat{w})=m_1.$$
Thus the proof.
\end{proof}
\begin{proof}[\textbf{Proof of Theorem \ref{4thm1} :}]  Theorem \ref{4thm1} is a consequence  of the Propositions \ref{4prop9}, \ref{4prop14} and \ref{4prop30}.
\end{proof}

\end{document}